\documentclass[12pt]{article}

\usepackage{ct}
\usepackage{amsthm}
\usepackage{enumerate}
\usepackage{color}
\usepackage{url}
\usepackage{multirow}
\usepackage[capitalize,nameinlink]{cleveref}
\usepackage{graphics}
\usepackage[normalem]{ulem}

\allowdisplaybreaks[1]
\newcommand{\ra}[1]{\renewcommand{\arraystretch}{#1}}

\theoremstyle{plain}
\newtheorem{result}[theorem]{Result}

\newcommand{\CC}{\mathbb{C}}
\newcommand{\Z}{\mathbb{Z}}
\newcommand{\F}{\mathbb{F}}
\newcommand{\cW}{\mathcal{W}}

\newcommand{\s}{\sigma}
\newcommand{\la}{\lambda}
\newcommand{\ep}{\epsilon}
\newcommand{\zp}{\zeta_{p}}
\newcommand{\sm}{\setminus}
\newcommand{\es}{\emptyset}
\newcommand{\ol}{\overline}
\newcommand{\lan}{\langle}
\newcommand{\ran}{\rangle}
\newcommand{\wti}{\widetilde}
\newcommand{\wh}{\widehat}
\newcommand{\floor}[1]{{\left\lfloor#1\right\rfloor}}
\newcommand{\ceiling}[1]{{\left\lceil#1\right\rceil}}
\newcommand{\Tr}{{\rm Tr}}
\newcommand{\mr}{\multirow}
\newcommand{\mc}{\multicolumn}

\newcommand\scalemath[2]{\scalebox{#1}{\mbox{\ensuremath{\displaystyle #2}}}}

\long\def\symbolfootnote[#1]#2{\begingroup%
\def\thefootnote{\fnsymbol{footnote}}\footnote[#1]{#2}\endgroup}

\dateline{Apr 28, 2021}{Sep 9, 2021}{TBD}


\Copyright{The authors. Released under the CC BY license (International 4.0).}

\title{Packings of partial difference sets}

\author[1]{Jonathan Jedwab\thanks{supported by an NSERC Discovery Grant.}}
\author[2]{Shuxing Li\thanks{supported by a PIMS Postdoctoral Fellowship.}}

\affil[1,2]{Department of Mathematics, Simon Fraser University, 8888 University Drive, Burnaby BC V5A 1S6, Canada\newline
\email{jed@sfu.ca}, \email{shuxing\_li@sfu.ca}%
}

\begin{document}

\maketitle

\begin{abstract}
A packing of partial difference sets is a collection of disjoint partial difference sets in a finite group~$G$.
This configuration has received considerable attention in design theory, finite geometry, coding theory, and graph theory over many years, although often only implicitly.
We consider packings of certain Latin square type partial difference sets in abelian groups having identical parameters,
the size of the collection being either the maximum possible or one smaller.
We unify and extend numerous previous results in a common framework, recognizing that a particular subgroup reveals important structural information about the packing. Identifying this subgroup allows us to formulate a recursive lifting construction of packings in abelian groups of increasing exponent, as well as a product construction yielding packings in the direct product of the starting groups.
We also study packings of certain negative Latin square type partial difference sets of maximum possible size in abelian groups, all but one of which have identical parameters, and show how to produce such collections using packings of Latin square type partial difference sets.
\end{abstract}

\section{Introduction}

A recurring theme across many diverse areas of mathematics is the natural occurrence of structures for which a key aspect can take one of only two values. We mention several examples of this two-valued phenomenon.
In coding theory, a projective two-weight code is a linear code whose codeword weights take one of exactly two distinct values \cite{CK}.
In finite geometry, a projective two-intersection set is a point set in projective space whose intersection with each hyperplane has one of exactly two distinct sizes \cite{CK};
and an $m$-ovoid or a tight set in a polar space is a point set whose intersection with every tangent hyperplane of the polar space has one of exactly two distinct sizes \cite[Chapter 2]{BvM}, \cite[Section 4.5]{MWX}.
In graph theory, a strongly regular graph has exactly two distinct eigenvalues \cite[Chapter 1]{BvM}.
In group theory, a rank $3$ permutation group has exactly two nontrivial orbitals \cite[Section 10]{CK}; and a symmetric Schur ring of dimension~$3$ over a group $G$ is a partition of the nonidentity group elements into exactly two nontrivial subsets which, together with the identity element, form a subalgebra of $\CC[G]$ \cite{Wie}.
In number theory, uniform cyclotomy is a partition of the nonzero elements of a finite field into cyclotomy classes such that exactly two distinct cyclotomic numbers occur~\cite{BMW}.

An underlying connection between many instances of these two-valued phenomena is provided by a partial difference set.
Indeed, let $D$ be a $k$-subset of an additive abelian group $G$ of order~$v$. The subset $D$ is a $(v,k,\la,\mu)$ \emph{partial difference set} in $G$ if the multiset $\{x-y \mid x,y \in D, x \ne y \}$ contains each nonidentity element of $D$ exactly $\la$ times and each nonidentity element of $G \sm D$ exactly $\mu$ times.
Partial difference sets, possibly having some additional properties, give rise to examples of each of the above structures. The construction of partial difference sets is therefore of great interest. We refer to \cite{CK,Ma} for excellent surveys of partial difference sets and equivalent structures, and to \cite{A,ACK,BKLP,BLMX,BDMR,BL,BWX,B,CGT,CM,CTZ,CF13,CP11,CP13,CRX,CHOP,CDMPS,CoM,CoMa,CoMP,CP14,CP17,CP18,CoPe13,D,DX00,DX04,FMX15JCT,FMX15C,FWXY,FX,FKM,HP,H00,H02,H03,HLX,HN,IZZ,LM95,LM96,MS,M14,M18,M21,MO,MX14,MX18,O16,O16DCC,O16JG,OP,P,P08EJC,P09JCD,P09DCC,P10,P19,PDS,P02,P08DCC,RX,SW,TPF,WQWZ} for investigations spanning twenty-five years into which groups $G$ contain a partial difference set with specified parameters $(v,k,\la,\mu)$.

This paper is concerned with the construction of two of the richest classes of partial difference sets, whose parameters are determined by only two integers. A partial difference set whose parameters $(v,k,\la,\mu)$ take the form
$\big(n^2,r(n-1), n+r(r-3),r(r-1)\big)$ has \emph{$(n,r)$ Latin square type},
and one whose parameters take the form
$\big(n^2,r(n+1), -n+r(r+3),r(r+1)\big)$ has \emph{$(n,r)$ negative Latin square type}.
A partial difference set $D$ in $G$ is \emph{regular} if $1_G \notin D$ and $D = \{d^{-1} \mid d \in D\}$.
A great deal of research has been conducted into the existence of a collection of $t > 1$ disjoint regular $(tc,c)$ Latin square type partial difference sets in an abelian group $G$ of order~$t^2c^2$.
The number of elements of $G$ avoided by the partial difference sets of such a collection is $tc$, so the collection has maximum possible size when $c>1$.
The principal motivation for this paper was the discovery of a widespread structural property that does not seem to have been previously recognized: in all previous constructions, possibly after some modification, the $tc$ avoided elements of $G$ form a subgroup~$U$. We therefore refer to such a collection of partial difference sets as a \emph{$(c,t)$ LP-packing in $G$ relative to~$U$}.
The presence of the associated subgroup $U$ is key to the recursive lifting construction of LP-packings of increasing exponent that we present in Section~\ref{sec5}. In general, lifting constructions (that increase the group exponent) have historically been more challenging to find than product constructions (that combine objects into the direct product of the starting groups).
Table~\ref{tab-LPpacking} illustrates how widely the concept of LP-packings has been previously studied, often implicitly, and in each case identifies the subgroup $U$ explicitly.
In some cases, identifying the crucial subgroup $U$ from the original reference requires considerable effort.

We shall unify and extend many results for LP-packings, whose original derivation relied on a variety of sometimes delicate approaches, by means of a common framework that depends only on elementary methods.
We shall show that we have considerable control over the choice of the associated subgroup~$U$.
In addition to the recursive construction of Section~\ref{sec5}, we shall give a product construction for LP-packings in Section~\ref{sec4}.
We present the following result (to be proved as Corollary~\ref{coro-LPpackingeven}) as a showcase for our constructions.

\begin{theorem}\label{thm-preview}
Let $p$ be prime, let $s_1,\dots,s_v$ be nonnegative integers (not all zero), and let $m=\min\{s_i \mid s_i>0\}$. For each $i=1,\dots,v$, let $G_i=\Z_{p^i}^{2s_i}$ and let $U_i$ be a subgroup of $G_i$ of order~$p^{is_i}$.
Let $G = \prod_{i=1}^v G_i$, let $U = \prod_{i=1}^v U_i$, and let $n = \sqrt{|G|}$.
Then for each $j = 0,\dots,m-1$, there exists an
$\big(\frac{n}{p^{m-j}}, p^{m-j}\big)$ LP-packing in $G$ relative to~$U$:
a collection of $p^{m-j}$ disjoint regular $\big(n, \frac{n}{p^{m-j}}\big)$ Latin square type partial difference sets in $G$ avoiding~$U$.
\end{theorem}

We shall show in Section~\ref{sec3} that a collection of $t$ disjoint regular $(tc,c)$ Latin square type partial difference sets with $c=1$ does not have maximum possible size, because there exists a size $t+1$ collection known equivalently as a $(t,t+1)$ partial congruence partition. However, we regard the $(1,t)$ LP-packing as a more natural object than the larger partial congruence partition, because it forms the base case of the recursive construction of Section~\ref{sec5} and so allows a powerful generalization to nonelementary abelian groups.

We similarly refer to a collection of $t-1 > 0$ disjoint regular $(tc,c)$ negative Latin square type partial difference sets in an abelian group $G$ of order~$t^2c^2$,
for which the $(c-1)(tc+1)$ nonidentity avoided elements of $G$ form a regular $(tc,c-1)$ negative Latin square type partial difference set in $G$, as a \emph{$(c,t-1)$ NLP-packing}; such a collection has maximum possible size for all $c \ge 1$.
The NLP-packing structure was already known, and previous constructions are summarized in Table~\ref{tab-NLPpacking}.
We shall use a product construction in Section~\ref{sec6} to combine small examples of NLP-packings with the families of LP-packings constructed recursively here, and so extend previous constructions of families of NLP-packings.
We mention that both LP-packings and NLP-packings give rise to group-based amorphic association schemes~\cite{vM}.

Our principal objective is to determine for which abelian groups $G$ and subgroups $U$ there exists a $(c,t)$ LP-packing in $G$ relative to $U$, and for which abelian groups $G$ there exists a $(c,t-1)$ NLP-packing in~$G$. We propose that the natural framework for studying regular $(tc,c)$ Latin square type and negative Latin square type partial difference sets is as these respective collections, rather than as single examples.

The rest of the paper is organized as follows.
Section~\ref{sec2} gives a historical overview of packings of partial difference sets of Latin square type and negative Latin square type in abelian groups.
Section~\ref{sec3} introduces LP-packings, describes their relationship to partial congruence partitions, establishes constraints on their structure, and presents a product construction.
Section~\ref{sec4} introduces an LP-partition as an auxiliary configuration in the construction of an LP-packing in a larger group from an LP-packing in a smaller group.
Section~\ref{sec5} recursively constructs infinite families of LP-partitions and LP-packings of increasing exponent, inspired by a recursive lifting construction of difference sets using relative difference sets~\cite{DJ}.
Section~\ref{sec6} introduces NLP-packings, and uses a product construction to combine an NLP-packing with an LP-packing to give an NLP-packing in the direct product of the starting groups.
Section~\ref{sec7} examines how LP-packings and NLP-packings are related to hyperbolic and elliptic strongly regular bent functions, and
Section~\ref{sec8} examines how these packings are related to reversible Hadamard difference sets.
Section~\ref{sec9} proposes some open problems for future research.

\begin{table}[ht]
\ra{1.3}
\caption{Known $(c,t)$ LP-packings in $G$ relative to $U$, where $p$ is prime and $\s$ is defined in~\eqref{eqn-s}}
\label{tab-LPpacking}
\begin{center}
\scalebox{0.86}{
\begin{tabular}{|c|c|c|c|c|c|}
\hline
$c$				& $t$ 			& Group $G$ 					& Subgroup $U$ 					& Restrictions 				& Source 				\\ \hline
\mr{2}{*}{$p^{w-s}$}		& \mr{2}{*}{$p^s$} 	& \mr{2}{*}{$\Z_{p^2}^w$} 			& \mr{2}{*}{$(p\Z_{p^2})^w$}			& $s$ is proper				& \mr{2}{*}{\cite[Thm.~3.1]{CRX}} 	\\
				&			&						&						& divisor of $w>1$			& 					\\ \hline
$p$				& $p$ 			& $\Z_{p^2}^2$ 					& $(p\Z_{p^2})^2$ 				&  					& \cite[Thm.~3.1]{D} 			\\ \hline
$2^{a-1}$			& $2$			& $\Z_{2^a}^2$					& order $2^a$, noncyclic			& 					& \cite[Thm.~3.1]{DP}	 		\\ \hline
$p^{w-\s(p,w)}$			& $p^{\s(p,w)}$		& $\Z_{p^2}^w$ 					& $(p\Z_{p^2})^w$ 				& $w > 1$ 				& \cite[Lem.~3.1]{HLX} 			\\ \hline
$\frac{n}{p^s}$ 		& $p^s$			& $\Z_{p^{a+1}}^{2us} \times \Z_{p^{a}}^{2ws}$	& $\Z_{p^{a+1}}^{us} \times \Z_{p^{a}}^{ws}$	& $n=\sqrt{|G|}$ and $a \ge 1$		& \cite[Thm.~3.1]{LM90}			\\ \hline
$p^{w-1}$			& $p$ 			& $\Z_{p^2}^w$ 					& $(p\Z_{p^2})^w$				& $w > 1$ 				& \cite[Thm.~2.2]{LM96} 		\\ \hline
$p^s$				& $p^s$ 		& $\Z_{p^2}^{2s}$ 				& $(p\Z_{p^2})^{2s}$ 				& $s > 1$  				& \cite[Prop.~4.1]{P02} 		\\ \hline
\mr{2}{*}{$p^{(2a-1)s}$}	& \mr{2}{*}{$p^s$}	& \mr{2}{*}{$\Z_{p^{2a}}^{2s}$} 		& \mr{2}{*}{$(p^a\Z_{p^{2a}})^{2s}$}		& 					& \cite[Lem.~4.3]{P06} 			\\
           			&			& 						& 						& 					& \cite[Thm.~3.1]{P08DCC} 		\\ \hline
\mr{2}{*}{$\frac{n}{p}$}	& \mr{2}{*}{$p$}	& \mr{2}{*}{$\Z_p^{2u} \times
							   \prod_{i=1}^v\Z_{p^{2i}}^{2s_{2i}}$}		& \mr{2}{*}{$\Z_p^u \times
													   \prod_{i=1}^v (p^i\Z_{p^{2i}})^{2s_{2i}}$}	& $n=\sqrt{|G|}$ and 			& \mr{2}{*}{\cite[Cor.~6.2]{P08DCC}}	\\
				&			&						&						& $u, s_{2i} \ge 0$			&					\\ \hline
$p^{w-1}$			& $p$ 			& $\Z_{p^2}^w$					& $(p\Z_{p^2})^w$				& $p$ is odd divisor of $w$		& \cite[Thm.~4.1]{RX}			\\ \hline
$2^{\ceiling{w/2}}$		& $2^{\floor{w/2}}$	& $\Z_{4}^w$ 					& $(2\Z_{4})^w$ 				& $w > 1$ 				& \cite[Lem.~4.5]{XD} 			\\ \hline
\end{tabular}
}
\end{center}
\end{table}

\begin{table}[ht]
\ra{1.3}
\caption{Known $(c,t-1)$ NLP-packings in $G$}
\label{tab-NLPpacking}
\begin{center}
\scalebox{0.84}{
\begin{tabular}{|c|c|c|c|c|}
\hline
$c$			& $t-1$ 		& Group $G$ 										& Restrictions 				& Source 				\\ \hline
\mr{3}{*}{$\frac{n}{2}$}& \mr{3}{*}{$1$} 	& \mr{3}{*}{$\Z_2^{2u} \times \Z_4^w \times \prod_{i=3}^v \Z_{2^i}^{2s_i} \times
						  \Z_3^{2y} \times \prod_{i=1}^{z} \Z_{p_i}^{4t_i}$} 					& $n=\!\sqrt{|G|}$ and prime $p_i > 3$,	& \mr{3}{*}{\cite[Prop.~3.12]{CP13}} 	\\				
			&			& 											& $u,w,s_i,y,t_i \ge 0$, 		& 					\\
			&			&											& if $u=0$ then $y=t_i=0$ 		&					\\ \hline
\mr{2}{*}{$\frac{n}{3}$}& \mr{2}{*}{$2$} 	& \mr{2}{*}{$\Z_3^{2+2u} \times \Z_9^w \times \prod_{i=3}^v \Z_{3^i}^{2s_i}$}		& $n=\sqrt{|G|}$ and $u,w,s_i \ge 0$,	& \mr{2}{*}{\cite[Prop.~3.13]{CP13}} 	\\
			& 			&      		    									& $w \ne 1$   			&						\\ \hline
\mr{2}{*}{$\frac{n}{4}$}& \mr{2}{*}{$3$} 	& $\Z_2^{4+2u} \times \Z_4^w \times \prod_{i=3}^v \Z_{2^i}^{4s_i}$			& $n=\sqrt{|G|}$ and $u,w,s_i \ge 0$,	& \mr{2}{*}{\cite[Prop.~3.14]{CP13}} 	\\ \cline{3-3}
			&			& $\Z_2^{2u} \times \Z_4^{w+2} \times \prod_{i=3}^v \Z_{2^i}^{4s_i}$			& $u \ne 1$ and $w \notin \{1,3\}$	& 					\\ \hline
\mr{2}{*}{$4^{w-1}$}	& \mr{2}{*}{$3$} 	& \mr{2}{*}{$\Z_{4}^{2u}\times \Z_2^{4w-4u}$} 						& $0 \le u \le w$ if $w$ odd 		& \cite[Cor.~2.2]{DX04} 		\\
			&			&											& $0 \le u < w$ if $w$ even		& \cite[Thm.~3.6]{DX06} 		\\ \hline
\mr{2}{*}{$\frac{n}{4}$}& \mr{2}{*}{$3$} 	& \mr{2}{*}{$\Z_2^{4u} \times \Z_4^{2w} \times \prod_{i=2}^v \Z_{2^{2i}}^{4s_{2i}}$} 	& $n=\sqrt{|G|}$ and $u,w,s_{2i} \ge 0$,& \mr{2}{*}{\cite[Thm.~4.2]{P08DCC}} 	\\
			&			& 											& $u+w \ge 1$   			&  					\\ \hline
$\frac{n}{3}$		& $2$			& $\Z_3^{2u+2} \times \prod_{i=1}^v \Z_{3^{2i}}^{2s_{2i}}$				& $n=\sqrt{|G|}$ and $u, s_{2i} \ge 0$ 	& \cite[Cor.~5.1]{P08DCC} 		\\ \hline
$3^{w-1}$		& $2$ 			& $\Z_{3}^{2w}$ 									& $w \ge 1$ 				& \cite[Thm.~1]{TPF} 			\\ \hline
\end{tabular}
}
\end{center}
\end{table}

\section{Historical Overview}\label{sec2}

In this section, we give an overview of previous constructions of regular Latin square type and negative Latin square type partial difference sets in abelian groups. In many cases, these results produce only a single partial difference set, whereas our objective in later sections will be to construct disjoint collections.

Let $G$ be an abelian group.
We firstly present some basic results involving the group ring $\Z[G]$ and character theory.
For $A \in \Z[G]$ and $g \in G$, denote the coefficient of $g$ in $A$ by~$[A]_g$ (so that $A = \sum_{g \in G} [A]_g g$).
For $A \in \Z[G]$, write $A^{(-1)}=\sum_{g \in G} [A]_gg^{-1}$.
For a subset $A$ of~$G$, by a standard abuse of notation we also denote the group ring element $\sum_{g \in A}g$ by $A$.
A \emph{character} $\chi$ of $G$ is a group homomorphism from $G$ to the multiplicative group of the complex field~$\mathbb{C}$. Write $\wh{G}$ for the character group of~$G$.
A character $\chi \in \wh{G}$ is \emph{principal} on a subgroup $H \leqslant G$ if $\chi(h)=1$ for each $h \in H$. For a subgroup $H \leqslant G$, write \[
H^{\perp}=\{ \chi \in \wh{G} \mid \mbox{$\chi$ is principal on $H$} \}
\]
(so that $G^\perp$ consists of only the principal character on~$G$).
For $\chi \in \wh{G}$ and $A \in \Z[G]$, write $\chi(A)$ for the character sum $\sum_{g \in G} [A]_g\chi(g)$. For a more detailed treatment of group rings and character theory, see \cite[Chap.~VI]{BJL} and \cite[Chap.~1]{Pott}.

\begin{definition}\label{def-pds}
Let $D$ be a $k$-subset of an additive group $G$ of order~$v$. The subset $D$ is a $(v,k,\la,\mu)$ \emph{partial difference set} (PDS) in $G$ if the multiset $\{x-y \mid x,y \in D, x \ne y \}$ contains each nonidentity element of $D$ exactly $\la$ times and each nonidentity element of $G-D$ exactly $\mu$ times.
The partial difference set is \emph{regular} if $1_G \notin D$ and $D=D^{(-1)}$.
\end{definition}
\noindent
The condition $D=D^{(-1)}$ in Definition~\ref{def-pds} is guaranteed to hold except in the special case $\la = \mu$ \cite[Prop.~1.2]{Ma}, in which case the partial difference set is a \emph{$(v,k,\la)$-difference set} in~$G$ (see Section~\ref{sec8}).
Provided that $D=D^{(-1)}$, the condition $1_G \notin D$ is not restrictive \cite[p.~222]{Ma}; we shall be concerned only with partial difference sets that are regular.
Let $D$ be a $k$-subset of a group $G$ of order $v$ for which $1_G \notin D$.
Then, in group ring notation, $D$ is a $(v,k,\la,\mu)$ PDS in $G$ if and only if
\begin{equation}\label{eqn-PDSgpring}
DD^{(-1)} = k-\mu + (\la-\mu) D + \mu G \quad \mbox{in $\Z[G]$}.
\end{equation}

The following result is a consequence of the orthogonality properties of characters.

\begin{proposition}[Fourier inversion formula] \label{prop-fourier}
Let $G$ be an abelian group and let $A \in \Z[G]$. Then
\[
[A]_g=\frac{1}{|G|}\sum_{\chi \in \wh{G}} \chi(A)\ol{\chi(g)} \quad \mbox{for each $g \in G$}.
\]
In particular, elements $A, B$ of $\Z[G]$ are equal if and only if $\chi(A)=\chi(B)$ for all $\chi \in \wh{G}$.
\end{proposition}
\noindent
For a regular PDS $D$, we can use the relation $D=D^{(-1)}$ to rewrite \eqref{eqn-PDSgpring} as $D^2 = k-\mu + (\la-\mu) D + \mu G$ in~$\Z[G]$. Applying a nonprincipal character $\chi$ to both sides then shows that the character sum $\chi(D)$ satisfies the quadratic equation $\chi(D)^2-(\la-\mu)\chi(D)-(k-\mu)=0$. Using Proposition~\ref{prop-fourier}, we therefore obtain the following characterization of a regular partial difference set in terms of its character sums.

\begin{lemma}[{\cite[Cor.~3.3]{Ma}}]\label{lem-pdschar}
Let $G$ be an abelian group of order $v$. Let $D$ be a $k$-subset of~$G$ for which $1_G \notin D$.
Let $\la,\mu$ be nonnegative integers satisfying
$k^2=k-\mu+(\la-\mu)k+\mu v$ and $(\la-\mu)^2+4(k-\mu) \ge 0$.
Then $D$ is a regular $(v,k,\la,\mu)$ partial difference set in $G$ if and only if
\[
\chi(D) = \tfrac{1}{2}\Big(\la-\mu \pm \sqrt{(\la-\mu)^2+4(k-\mu)}\Big)
\quad \mbox{for all nonprincipal characters $\chi$ of $G$}.
\]
\end{lemma}

We are interested in the following two types of parameters for partial difference sets.

\begin{definition}\label{def-LatinPDS}
Let $n \ge 1$ and $r \ge 0$ be integers.
An $(n^2,r(n-\ep),\ep n+r^2-3\ep r,r^2-\ep r)$ PDS
is an \emph{$(n,r)$ Latin square type PDS} if $\ep=1$, and is an \emph{$(n,r)$ negative Latin square type PDS} if $\ep=-1$.
\end{definition}
\noindent
The connection between Latin square type partial difference sets and Latin squares is ``rather indirect'' \cite[p.~2]{HN}.
We allow the case $r=0$ in Definition~\ref{def-LatinPDS}, which corresponds to the empty set (see Remark~\ref{rem-NLPpacking}~$(iii)$).

We next give a character-theoretic description of the two parameter sets of interest.

\begin{lemma}\label{lem-LatinPDS}
Let $n \ge 1$ and $r \ge 0$ be integers, and let $\ep \in \{1,-1\}$.
Let $G$ be an abelian group of order $n^2$, and let $D$ be an $r(n-\ep)$-subset of $G$. Then $D$ is a regular $(n,r)$ Latin square type PDS in $G$ (when $\ep=1$), and is a regular $(n,r)$ negative Latin square type PDS in $G$ (when $\ep=-1$), if and only if
\[
\chi(D) \in \{ -\ep r, \ep (n-r) \} \quad \mbox{for all nonprincipal characters $\chi$ of $G$}.
\]
If $D$ is a regular $(n,r)$ Latin square type or negative Latin square type PDS in $G$, then the character set $\wti{D}=\{ \chi \in \wh{G} \mid \chi(D)=\ep(n-r) \}$ is also a regular $(n,r)$ Latin square type or negative Latin square type PDS in~$\wh{G}$, respectively.
\end{lemma}
\begin{proof}
The first statement follows directly from Lemma~\ref{lem-pdschar}.
The second statement is obtained from \cite[Thm.~3.4]{Ma}: let $D^+$ be the dual of $D$ as defined in \cite{Ma}, and note that $\wti{D} = D^+$ in the case $\ep=1$ and that $\wti{D} = \wh{G}-1_{\wh{G}}-D^+$ in the case $\ep=-1$.
\end{proof}

We now briefly survey the numerous previous constructions of Latin square type and negative Latin square type partial difference sets, which we classify into three classes that are restricted to elementary abelian groups and five that are not.
The three construction classes that are restricted to elementary abelian groups arise from the following sources.

\begin{enumerate}[(1)]
\item \emph{Projective two-intersection sets}.

Projective two-intersection sets (sometimes called two-character sets) are classical configurations from finite geometry that provide a series of constructions \cite{A,ACK,BL,CHOP,CDMPS,CoM,CoMa,CoMP,CoPe13,HP,IZZ,P} (see also \cite[Sect.~9]{Ma}).
Recently, projective two-intersection sets have been intensively studied in the context of intriguing sets in polar spaces (see \cite{BKLP,BLMX,BDMR,B,CP14,CP17,CP18,FMX15JCT}, for example). See \cite{C} for a recent survey of intriguing sets, and \cite[Sect.~4.5]{MWX} for a detailed account of the connection between these sets and Latin square type and negative Latin square type PDSs.

\item \emph{Nondegenerate quadratic forms and their generalizations}.

A fundamental construction employs nondegenerate quadratic forms over finite fields \cite[Thm.~2.6]{Ma}. This construction has been greatly extended to allow the quadratic form to be replaced by other functions, including bent and vectorial bent functions \cite{CGT,CM,CTZ,CP13,FWXY,OP,TPF}.

\item \emph{Cyclotomic classes of a finite field}.

A seminal construction due to Baumert, Mills and Ward involves uniform cyclotomy in a finite field \cite{BMW}, \cite[Thm.~10.3]{Ma}. Further cyclotomic constructions have been proposed, based on sophisticated number-theoretic techniques \cite{M14,M18,MO,MX14,MX18,O16,O16JG}.
\end{enumerate}

We now describe the five construction classes that are not restricted to elementary abelian groups.
We shall later use LP-packings and NLP-packings to streamline many of these constructions, using only elementary methods. In some cases, we shall give a unified treatment for previously known parameter sets; in other cases, we shall produce examples with entirely new parameter sets.

\begin{enumerate}[(1)]
\item \emph{Direct constructions via partial congruence partitions}

\begin{definition}\label{def-PCP}
Let $G$ be a group of order~$n^2>1$. An \emph{$(n,r)$ partial congruence partition of degree~$r$ in $G$} is a collection $\{U_1,\dots,U_r\}$ of order $n$ subgroups of $G$ such that $U_i \cap U_j=\{1_G\}$ for all $i \ne j$.
\end{definition}

If $U_1, \dots, U_r$ is an $(n,r)$ partial congruence partition in $G$, then $\sum_{i=1}^r (U_i-1_G)$ is a regular $(n,r)$ Latin square type partial difference set in $G$ \cite[Sect.~4.1]{BJ}.
A central objective in the study of $(n,r)$ partial congruence partitions in groups of order~$n^2$ is to determine the largest possible degree~$r$ \cite{BJ,J89a,J89b}. The following result determines this value for abelian $p$-groups of the form~$H \times H$.

\begin{result}[{\cite[Thm.~2.7]{BJ}}]\label{res-PCPPDS}
Let $s_1,\dots,s_v$ be nonnegative integers and let $m = \min \{s_i \mid s_i > 0\}$. Let $p$ be prime and let $G=\prod_{i=1}^v\Z_{p^i}^{2s_i}$ and let $n = \sqrt{|G|}$.
Then the largest integer~$r$ for which there exists an $(n,r)$ partial congruence partition in $G$ is~$p^m+1$.
Therefore there exists a regular $(n,r)$ Latin square type PDS in $G$ for each positive integer $r \le p^m+1$.
\end{result}

We shall strengthen Result~\ref{res-PCPPDS} using Corollary~\ref{coro-LPpackingeven}: see Remark~\ref{rem-PDSeven}~$(i)$.

\item \emph{Direct constructions via finite local rings}

For positive integers $p$ and $w$, define the function
\begin{equation}\label{eqn-s}
\s(p,w)=\max_{\ell \mid w} \ell \left(\floor{\frac{w/\ell-2}{p}}+1\right),
\end{equation}
which determines the parameters of Latin square type partial difference sets in $\Z_{p^2}^w$ under the construction of \cite[Corollary 3.4]{HLX}. Various direct constructions of Latin square type partial difference sets exploit the structure of finite local rings \cite{CRX,DX04,H02,H03,HLX,HN,LM90,LM96,P02,RX}. We summarize below the constructions from \cite{CRX,H02,H03,HLX,HN,LM90,LM96,RX}, which we shall strengthen using Corollaries~\ref{coro-LPpackingeven} and~\ref{coro-LPpackingodd}:
see Remarks~\ref{rem-PDSeven}~$(ii)$ and~\ref{rem-PDSodd}.

\begin{result}\label{res-ringPDS}
Let $p$ be prime.
There exists an $(n,r)$ Latin square type PDS in a group $G$ of order $n^2$ for the following $G$ and $r$, where $r \le n$.
\begin{enumerate}[(i)]

\item $G=\Z_{p^2}^w$ for $w > 1$, and $r$ is a positive integer multiple of $\frac{n}{p^{\s(p,w)}}$ {\rm (\cite[Thm.~3.2, Cor.~3.4]{HLX})}.

\item $G=\Z_{p^2}^w \times \Z_{p}^{2uw}$ for $w > 1$ and a positive integer~$u$, and $r$ is a positive integer multiple of $\frac{n}{p^{\s(p,w)}}$ {\rm (\cite[Thm.~3.1]{H02})}.

\item $G = \prod_{i=1}^v \Z_{p^i}^{2s_i}$ for nonnegative integers $s_1,\dots,s_v$, and $r$ is a positive integer multiple of $\frac{n}{p^{\,\gcd(s_1,\ldots,s_v)}}$ {\rm (\cite[Thm.~3.4, Cor.~3.5]{H03})}.
\end{enumerate}
\end{result}

\item \emph{Lifting constructions via finite local rings}

Finite local rings also provide constructions for Latin square type partial difference sets that lift examples from a group $G \times G$ to a larger group $G' \times G'$ \cite{H00,H03,HN}.
We shall present in Theorem~\ref{thm-composition} a ring-free lifting construction that applies to a collection of Latin square type partial difference sets, rather than just a single one.
This construction leads to the result stated in Theorem~\ref{thm-preview}, which recovers the parameters of partial difference sets constructed via lifting in each of \cite{H00}, \cite[Sects.~4,~5]{H03}, and \cite{HN}.

\item \emph{Product constructions}

Various delicate product constructions for Latin square type and negative Latin square type partial difference sets have been found \cite{M14,P08EJC,P09JCD,P19,PDS,P08DCC}. The following examples occur in $p$-groups having arbitrarily large exponent.

\begin{result}[{\cite[Cor.~6.2]{P08DCC}}]\label{res-PDSproductLP}
Let $p$ be prime, and let $u$ and $s_2,s_4,\dots,s_{2v}$ be nonnegative integers.
Let $G = \Z_p^{2u} \times \prod_{i=1}^v\Z_{p^{2i}}^{2s_{2i}}$ and let $n = \sqrt{|G|}$.
Then there is a partition of $G-1_G$ into $p$ regular partial difference sets in~$G$, of which $p-1$ are of $(n,\frac{n}{p})$ Latin square type and one is of $(n,\frac{n}{p}+1)$ Latin square type.
\end{result}

We shall strengthen Result~\ref{res-PDSproductLP} using Corollary~\ref{coro-LPpackingeven} (see Remark~\ref{rem-PDSeven}~$(v)$).

\begin{result}[{\cite[Thm.~4.2, Cor.~5.1]{P08DCC}}]\label{res-PDSproductNLP}
\quad
\begin{enumerate}[$(i)$]

\item
Let $u,w$ and $s_4,s_6,\dots,s_{2v}$ be nonnegative integers satisfying $u+w \ge 1$.
Let $G = \Z_2^{4u} \times \Z_4^{2w} \times \prod_{i=2}^v \Z_{2^{2i}}^{4s_{2i}}$ and let $n = \sqrt{|G|}$.
Then there is a partition of $G -1_G$ into $4$ regular partial difference sets in~$G$,
of which three are of $(n,\frac{n}{4})$ negative Latin square type and one is of $(n,\frac{n}{4}-1)$ negative Latin square type.

\item
Let $u$ and $s_2,s_4,\dots,s_{2v}$ be nonnegative integers.
Let $G = \Z_3^{2u+2} \times \prod_{i=1}^v \Z_{3^{2i}}^{2s_{2i}}$ and let $n = \sqrt{|G|}$.
Then there is a partition of $G -1_G$ into $3$ regular partial difference sets in~$G$, of which two are of $(n,\frac{n}{3})$ negative Latin square type and one is of $(n,\frac{n}{3}-1)$ negative Latin square type.

\end{enumerate}
\end{result}

We shall strengthen Result~\ref{res-PDSproductNLP} using Corollary~\ref{coro-NLPpacking} (see Remark~\ref{rem-PDSproduct}).

\item \emph{Constructions in Galois domains}

A final construction class uses cyclotomy \cite{FKM} and character sums \cite{M21,O16DCC} over Galois domains (direct products of finite fields) to produce partial difference sets in the direct product of elementary abelian groups.

\end{enumerate}

\section{LP-Packings}\label{sec3}

In this section, we introduce a $(c,t)$ LP-packing as a collection of $t$ disjoint regular $(tc,c)$ Latin square type PDSs, in an abelian group of order $t^2c^2$, whose union satisfies an additional property. We provide examples of LP-packings, describe their relationship to partial congruence partitions, establish constraints on their structure, and present a product construction.

A principal motivation for this paper was the realization that in many previous constructions of such collections,
possibly after some modification as summarized in Table~\ref{tab-LPpacking}, the $tc$-set $G-\sum_{i=1}^t P_i$ of elements avoided by the partial difference sets forms a subgroup $U$ of~$G$.
Our definition of a $(c,t)$ LP-packing is based on this observation.

\begin{definition}[LP-packing]\label{def-LPpacking}
Let $t>1$ and $c>0$ be integers. Let $G$ be an abelian group of order~$t^2c^2$, and let $U$ be a subgroup of $G$ of order~$tc$. A $(c,t)$ \emph{LP-packing in $G$ relative to $U$} is a collection $\{P_1,\dots,P_t\}$ of $t$ regular $(tc,c)$ Latin square type PDSs in $G$ for which $\sum_{i=1}^t P_i = G-U$.
\end{definition}

\begin{remark}
In Definition~\ref{def-LPpacking}, each $P_i$ is a $c(tc-1)$-subset of $G$, and $|G-U| = tc(tc-1)$, so the relation $\sum_{i=1}^t P_i = G-U$ is equivalent to the statement that the $P_i$ are disjoint and their union is~$G-U$.
\end{remark}

We next give a characterization of a $(c,t)$ LP-packing involving character sums.

\begin{lemma}\label{lem-LPpacking}
Let $P_1,\dots,P_t$ be $c(tc-1)$-subsets of an abelian group $G$ of order~$t^2c^2$ and let $U$ be a subgroup of $G$ of order~$tc$.
Then $\{P_1,\dots,P_t\}$ is a $(c,t)$ LP-packing in $G$ relative to~$U$ if and only if the multiset equality
\begin{equation}
\label{eqn-U}
\{\chi(P_1),\dots,\chi(P_t)\} = \begin{cases}
  \{-c,\dots, -c\} 	    & \mbox{if $\chi \in U^\perp$,} \\
  \{(t-1)c, -c, \dots, -c\} & \mbox{if $\chi \notin U^\perp$}
\end{cases}
\end{equation}
holds for all nonprincipal characters $\chi$ of~$G$.
\end{lemma}
\begin{proof}
Let $\chi$ be a nonprincipal character of~$G$.

Suppose firstly that $\{P_1,\dots,P_t\}$ is a $(c,t)$ LP-packing in $G$ relative to~$U$. Apply $\chi$ to the equation $\sum_{i=1}^t P_i = G-U$ to give
\[
\sum_{i=1}^t \chi(P_i) = \begin{cases}
  -tc & \mbox{if $\chi \in U^\perp$,} \\
  0   & \mbox{if $\chi \notin U^\perp$.}
\end{cases}
\]
Since $\chi(P_i) \in \{-c, (t-1)c\}$ for each $i$ by Lemma~\ref{lem-LatinPDS}, this implies that
\[
\{\chi(P_1),\dots,\chi(P_t)\} = \begin{cases}
  \{-c,\dots, -c\} 	    & \mbox{if $\chi \in U^\perp$,} \\
  \{(t-1)c, -c, \dots, -c\} & \mbox{if $\chi \notin U^\perp$},
\end{cases}
\]
as required.

Now suppose that \eqref{eqn-U} holds. Then each $P_i$ is a $c(tc-1)$-subset of $G$ for which $\chi(P_i) \in \{-c, (t-1)c\}$, and so by Lemma~\ref{lem-LatinPDS} is a regular $(tc,c)$ Latin square type PDS in~$G$. It remains to show that $\sum_{i=1}^t P_i = G-U$. We have from $|P_i|=c(tc-1)$ and \eqref{eqn-U} that, for all $\chi \in \wh{G}$,
\[
\chi \Big(\sum_{i=1}^t P_i\Big) = \begin{cases}
  tc(tc-1) & \mbox{if $\chi \in G^\perp$}, \\
  -tc      & \mbox{if $\chi \in U^\perp \sm G^\perp$,} \\
  0        & \mbox{if $\chi \notin U^\perp$.}
\end{cases}
\]
By the Fourier inversion formula (Proposition~\ref{prop-fourier}), this implies that
$\sum_{i=1}^t P_i = G-U$, as required.
\end{proof}

A $(t,t+1)$ partial congruence partition (see Definition~\ref{def-PCP}) lies at the heart of many configurations {\rm\cite{Dillon}, \cite[Sect.~4]{ES}}. We next show that this structure can equivalently be expressed as a $(1,t)$ LP-packing.

\begin{proposition}\label{prop-PCP}
Let $G$ be an abelian group of order $t^2$, let $U_1, \dots, U_t$ be $t$-subsets of~$G$, and let $U_0$ be a subgroup of $G$ of order~$t$. Then $\{ U_0,U_1,\dots,U_t\}$ is a $(t,t+1)$ partial congruence partition in $G$ if and only if
$\{ U_1-1_G,\dots,U_t-1_G\}$ is a $(1,t)$ LP-packing in $G$ relative to~$U_0$.
\end{proposition}
\begin{proof}
By Definition~\ref{def-PCP},
$\{ U_0,U_1,\dots,U_t\}$ is a $(t,t+1)$ partial congruence partition in $G$ if and only if each of $U_1,\dots,U_t$ is a subgroup of $G$ and $\sum_{i=0}^t U_i = G + t 1_G$.
By Definitions~\ref{def-LatinPDS} and~\ref{def-LPpacking},
$\{ U_1-1_G,\dots,U_t-1_G\}$ is a $(1,t)$ LP-packing in $G$ relative to~$U_0$
if and only if each of $U_1-1_G,\dots,U_t-1_G$ is a regular $(t^2,t-1,t-2,0)$ PDS in $G$ and $\sum_{i=1}^t (U_i-1_G) = G-U_0$. It is therefore sufficient to show for $i=1,\dots,t$ that $U_i$ is a subgroup of $G$ if and only if $U_i-1_G$ is a regular $(t^2,t-1,t-2,0)$ partial difference set in~$G$. This follows directly from Definition~\ref{def-pds}.
\end{proof}

The following result greatly constrains the degree of a partial congruence partition.

\begin{result}\label{res-PCPdegree}
Suppose there exists an $(n,r)$ partial congruence partition of degree $r$ in an abelian group $G$ of order~$n^2$. Then $r \le n+1$, and
\begin{enumerate}[(i)]
\item
if $r \ge \floor{\sqrt{n}}+2$ then $G$ is elementary abelian {\rm(\cite[Thms.~3.1,~3.4]{J89a})}.

\item
if $r=\sqrt{n}+1$ then either $G$ is elementary abelian or $G=(\Z_{p^2})^w$ for some prime $p$ and positive integer~$w$ {\rm(\cite[Cor.~5.3]{J89b})}.
\end{enumerate}
\end{result}

By Result~\ref{res-PCPdegree}~$(i)$, a $(t,t+1)$ partial congruence partition in an abelian group~$G$ can exist only when $t = p^s$ for some prime $p$ and positive integer~$s$, and $G \cong \Z_p^{2s}$. 
By Proposition~\ref{prop-PCP}, such a structure exists if and only if there is a $(1,p^s)$ LP-packing in $G$ relative to an order $p^s$ subgroup. 
We now show, in an example we shall refer to frequently in the rest of the paper, that this structure exists for all primes $p$ and positive integers~$s$ by reference to a spread of a vector space over a finite field.

\begin{definition}\label{def-spread}
Let $p$ be prime and $s$ a positive integer, and let $V$ be a $2s$-dimensional vector space over the field $\F_p$ with identity~$1_V$. A \emph{spread} of $V$ is a $(p^s+1)$-set $\{H_0,H_1,\dots,H_{p^s}\}$ of $s$-dimensional vector subspaces of $V$ such that $H_i \cap H_j = \{1_V\}$ for all $i \ne j$.
\end{definition} 

\begin{example}\label{ex-spread}
Let $p$ be prime and $s$ a positive integer, and let $V$ be a $2s$-dimensional vector space over the field $\F_p$ with identity~$1_V$.
Then there exists a spread $\{H_0, H_1, \dots, H_{p^s}\}$ of $V$ \cite[Chap.~VII, Sect.~5]{HP73}.
By Definitions~\ref{def-PCP} and~\ref{def-spread}, 
$\{ H_0, \ldots, H_{p^s} \}$ forms a $(p^s,p^s+1)$ partial congruence partition in~$V$, because $H_i$ is a subgroup of order $p^s$ in $V \cong \Z_p^{2s}$ if and only if $H_i$ is an $s$-dimensional vector subspace of~$V$.
Then by Proposition~\ref{prop-PCP} with $t=p^s$, we see that
$\{ H_1-1_V, \ldots, H_{p^s}-1_V \}$ is a $(1,p^s)$ LP-packing in $V \cong \Z_{p}^{2s}$ relative to~$H_0 \cong \Z_p^s$,
and by Lemma~\ref{lem-LPpacking}
each nonprincipal character of $V$ is principal on exactly one of $H_0,\dots,H_{p^s}$. 
\end{example}

\begin{remark}\label{rem-PCP}
According to Definition~\ref{def-LPpacking}, a $(c,t)$ LP-packing involves a collection of $t$ disjoint regular $(tc,c)$ Latin square type PDSs in an abelian group $G$ of order~$t^2c^2$. Since this collection covers all but $tc$ elements of $G$, and each PDS contains $c(tc-1)$ elements, the collection has maximum possible size when $c>1$. The collection is not necessarily of maximum possible size when $c=1$: for $p$ prime and $t=p^s$, the $(t,t+1)$ partial congruence partition of Example~\ref{ex-spread} is a collection of $t+1$ disjoint regular $(t,1)$ Latin square type PDSs in the elementary abelian group of order $t^2$.
However, as Result~\ref{res-PCPdegree} shows, if we instead attempt to create a partial congruence partition in a non-elementary abelian group of order $t^2$ then the degree must drop from $t+1$ to at most $\floor{\sqrt{t}}+1$, covering a proportion of only about $1/\sqrt{t}$ of the elements of~$G$.
In contrast, in Section~\ref{sec5} we shall provide constructions of $(c,t)$ LP-packings with $c>1$ in various nonelementary abelian groups.
For this reason, we regard a $(c,t)$ LP-packing with $c>1$ as a natural generalization of a $(t,t+1)$ partial congruence partition.
\end{remark}

The characterization in Lemma~\ref{lem-LPpacking}, together with Lemma~\ref{lem-LatinPDS}, shows that we can combine subsets from an LP-packing to obtain Latin square type PDSs with various parameters.
\begin{lemma}\label{lem-LPpackingtoPDS}
Suppose that $\{P_1, \ldots, P_t\}$ is a $(c,t)$ LP-packing in an abelian group $G$ of order $t^2c^2$ relative to a subgroup $U$ of order~$tc$. Let $I$ be a $b$-subset of $\{1,\dots,t\}$. Then
\begin{enumerate}[$(i)$]
\item $\sum_{i \in I} P_i$ is a regular $(tc,bc)$ Latin square type PDS in~$G$.
\item $\sum_{i \in I} P_i+U-1_G$ is a regular $(tc,bc+1)$ Latin square type PDS in~$G$.
\end{enumerate}
\end{lemma}

\begin{remark}\label{rem-LPpacking}
Suppose that $\{P_1, \ldots, P_t\}$ is a $(c,t)$ LP-packing in an abelian group $G$ of order $t^2c^2$ relative to a subgroup $U$ of order~$tc$.
\begin{enumerate}[(i)]

\item
By Definition~\ref{def-LPpacking} and Lemma~\ref{lem-LPpackingtoPDS}, we see that
$\{P_1,\dots,P_{t-1},P_t+U-1_G\}$ is a partition of $G-1_G$ into $t$ regular partial difference sets in~$G$, of which the first $t-1$ are of $(tc,c)$ Latin square type and the last is of $(tc,c+1)$ Latin square type.

\item
For each $i$, let $\wti{{P\mkern 0mu}_i} 
= \{\chi \in \wh{G} \mid \chi(P_i) = (t-1)c\}$.
By Lemma~\ref{lem-LatinPDS}, each $\wti{{P\mkern 0mu}_i}$ is a regular $(tc,c)$ Latin square type PDS in~$\wh{G}$, and by Lemma~\ref{lem-LPpacking} each $\chi \in \wh{G} \sm U^\perp$ belongs to exactly one~$\wti{{P\mkern 0mu}_i}$ and each $\chi \in U^\perp$ belongs to no~$\wti{{P\mkern 0mu}_i}$. By Definition~\ref{def-LPpacking}, the collection $\{\wti{P_1}, \ldots, \wti{{P\mkern 0mu}_t}\}$
is therefore a $(c,t)$ LP-packing in $\wh{G}$ relative to~$U^{\perp}$.
\end{enumerate}
\end{remark}

Combining the subsets of an LP-packing into equally-sized collections gives an LP-packing with fewer subsets.

\begin{lemma}\label{lem-LPpackingunion}
Suppose there exists a $(c,t)$ LP-packing in an abelian group $G$ of order $t^2c^2$ relative to a subgroup $U$ of order~$tc$, and suppose $s$ divides~$t$. Then there exists an $(sc,\frac{t}{s})$ LP-packing in $G$ relative to~$U$.
\end{lemma}
\begin{proof}
Let $\{P_1, \dots, P_t\}$ be a $(c,t)$ LP-packing in $G$ relative to~$U$, and let
\[
P'_i = P_{is+1} + P_{is+2} + \dots + P_{is+s} \quad \mbox{for $i = 0, 1,\dots, \frac{t}{s}-1$}.
\]
Then Lemma~\ref{lem-LPpacking} shows that $\{P'_0,\dots,P'_{\frac{t}{s}-1}\}$ is an $(sc,\frac{t}{s})$ LP-packing in $G$ relative to~$U$.
\end{proof}

The following result, which is inspired by \cite[Lemma 3.2]{DP}, allows us to establish some constraints on LP-packings in Proposition~\ref{prop-LPpackingstruc}. These constraints assist in finding examples computationally, as demonstrated in Example~\ref{ex-LPpacking444}.

\begin{lemma}\label{lem-LSPDS}
Let $G$ be a group of order $n^2$ and let $U$ be a subgroup of order~$n$.
Let $\{g_0,g_1,\dots, $ $g_{n-1}\}$ be a complete set of right coset representatives for $U$ in $G$, where $Ug_0=U$. Suppose $P$ is a regular $(n,r)$ Latin square type PDS in $G$ for which $P \cap U = \es$. Then $|P \cap Ug_i|=r$ for each $i = 1, \dots, n-1$.

\end{lemma}
\begin{proof}
Since $P \cap U = \es$,
\begin{equation}\label{eqn-|P|}
r(n-1) = |P| = \sum_{i=1}^{n-1} |P\cap U g_i|.
\end{equation}
For all $x,y \in P$, we have
$x y^{-1} \in U$ if and only if $x,y$ belong to the same right coset of $U$.
Therefore
\begin{equation}\label{eqn-gUPP}
\sum_{i=1}^{n-1} |P \cap U g_i|^2 = \sum_{g \in U}[PP^{(-1)}]_g.
\end{equation}
Now $P$ is an $(n^2,r(n-1),n+r^2-3r,r^2-r)$ PDS in $G$, and so by \eqref{eqn-PDSgpring}
\[
PP^{(-1)} = r(n-r) + (n-2r)P + (r^2-r)G.
\]
Since $P \cap U = \es$, this implies that
\[
\sum_{g \in U} [PP^{(-1)}]_g = r(n-r) + (r^2-r)n = r^2(n-1).
\]
Substitution in \eqref{eqn-gUPP} then gives
\[
\sum_{i=1}^{n-1} |P \cap U g_i|^2 = r^2(n-1),
\]
which, by the Cauchy-Schwarz inequality and \eqref{eqn-|P|}, gives
$|P \cap Ug_i|=r$ for each $i=1,\dots,n-1$ as required.
\end{proof}

\begin{proposition}\label{prop-LPpackingstruc}
Let $G$ be an abelian group of order $t^2c^2$, and let $U$ be a subgroup of $G$ of order~$tc$.
Let $\{g_0,g_1,\dots,g_{tc-1}\}$ be a complete set of right coset representatives for $U$ in $G$, where $Ug_0=U$, and define subsets $S,N$ of $\{1,\dots,tc-1\}$ by
\begin{align*}
& S = \{j \mid Ug_j = Ug_j^{-1}\}, \\
& \mbox{$N$ contains exactly one of $j,k$ when $Ug_j \ne Ug_j^{-1} = Ug_k$}.
\end{align*}
Suppose that $\{P_1,\dots,P_t\}$ is a $(c,t)$ LP-packing in $G$ relative to~$U$.
Then there are $c$-subsets $H_{ij}$ of $U$ for $i=1,\dots,t$ and $j=1,\dots,tc-1$ satisfying
\begin{align*}
P_i &=\sum_{j \in S} H_{ij}g_j + \sum_{j \in N} \big(H_{ij}g_j + H_{ij}^{(-1)}g_j^{-1}\big) \quad \mbox{for each $i$}, \\
H_{ij}^{(-1)}g_j^{-1} &= H_{ij}g_j \quad \mbox{for $j \in S$}, \\
H_{ij}^{(-1)}g_j^{-1} &= H_{ik}g_k \quad \mbox{for $j \in N$ where $Ug_j^{-1} = Ug_k$}, \\
U &= \sum_{i=1}^t H_{ij} \quad \mbox{for each~$j$}.
\end{align*}
\end{proposition}
\begin{proof}
By Definition~\ref{def-LPpacking}, each $P_i$ is a regular $(tc,c)$ Latin square type PDS in $G$, and $G-U$ is the disjoint union of the~$P_i$.
Write
\begin{align}
P_i &= \sum_{j=1}^{tc-1} (P_i \cap Ug_j) \nonumber \\
    &= \sum_{j \in S} (P_i \cap Ug_j) +
       \sum_{j \in N} \big((P_i \cap Ug_j) + (P_i \cap Ug_j^{-1})\big). \label{eqn-PiUgj}
\end{align}
By Lemma~\ref{lem-LSPDS}, for each $i = 1,\dots, t$ and each $j = 1,\dots,tc-1$ we may write
\[
P_i \cap U g_j = H_{ij} g_j \quad \mbox{for some $c$-subset $H_{ij}$ of~$U$}.
\]
Substitute in \eqref{eqn-PiUgj} and use $P_i = P_i^{(-1)}$ to give
\[
P_i = \sum_{j \in S} H_{ij} g_j + \sum_{j \in N} \big(H_{ij}g_j + H_{ij}^{(-1)}g_j^{-1}\big),
\]
and the constraints
$H_{ij}^{(-1)}g_j^{-1} = H_{ij}g_j$ for $j \in S$, and
$H_{ij}^{(-1)}g_j^{-1} = H_{ik}g_k$ for $j \in N$ where $Ug_j^{-1} = Ug_k$.
For each $j$, we have $Ug_j = \sum_{i=1}^t (P_i \cap Ug_j) = \sum_{i=1}^t H_{ij} g_j$ and therefore $U = \sum_{i=1}^t H_{ij}$.
\end{proof}

\begin{example}\label{ex-LPpacking444}
Using Proposition~\ref{prop-LPpackingstruc}, we obtain the following exhaustive search results 
for LP-packings in $\Z_4^3$ relative to two nonisomorphic order~8 subgroups.
\begin{enumerate}[$(i)$]
\item
There are exactly $1536$ distinct $(4,2)$ LP-packings $\{P_1,P_2\}$ in $\Z_4^3=\lan x,y,z \ran$ relative to $(2\Z_4) \times (2\Z_4) \times (2\Z_4)=\lan x^2,y^2,z^2 \ran$, one of which is
    \begin{align*}
    P_1=&(1+x^2+y^2+x^2y^2)x+(x^2+x^2y^2+z^2+y^2z^2)y+(x^2+x^2z^2+y^2+y^2z^2)z\\
    &+(y^2+z^2+x^2+x^2y^2z^2)xy+(1+x^2+z^2+x^2z^2)xz \\
    &+(x^2+x^2y^2+x^2z^2+x^2y^2z^2)yz+(y^2+z^2+x^2y^2+x^2z^2)xyz, \\
    P_2=&\lan x,y,z \ran-\lan x^2,y^2,z^2 \ran-P_1.
    \end{align*}
The presented form of $P_1$ shows that both $P_1$ and $P_2$ intersect the seven nonidentity cosets of the subgroup $\lan x^2,y^2,z^2 \ran$ in $\lan x,y,z \ran$ in exactly 4 elements.
\item
There are exactly $512$ distinct $(4,2)$ LP-packings $\{R_1,R_2\}$ in $\Z_4^3=\lan x,y,z \ran$ relative to $\Z_4 \times (2\Z_4)=\lan x,y^2 \ran$, one of which is
    \begin{align*}
    R_1=&(y^2+xy^2+x^2y^2+x^3)z+(x+xy^2+x^3+x^3y^2)z^2+(y^2+x+x^2y^2+x^3y^2)z^3\\
    &+(x+x^2+x^2y^2+x^3y^2)y+(1+x^2+x^3+x^3y^2)yz \\
    &+(x+x^2+x^2y^2+x^3y^2)yz^2+(1+x+xy^2+x^2)y^3z^3, \\
    R_2=&\lan x,y,z \ran-\lan x,y^2 \ran-R_1.
    \end{align*}
The presented form of $R_1$ shows that both $R_1$ and $R_2$ intersect the seven nonidentity cosets of the subgroup $\lan x,y^2 \ran$ in $\lan x,y,z \ran$ in exactly 4 elements.
\end{enumerate}
\end{example}

We conclude this section with a product construction for LP-packings.

\begin{theorem}\label{thm-LPpackingproduct}
For $j=1,2$, suppose there exists a $(c_j,t)$ LP-packing in an abelian group $G_j$ of order $t^2c_j^2$ relative to a subgroup $U_j$ of order~$tc_j$. Then there exists a $(tc_1c_2,t)$ LP-packing in $G_1 \times G_2$ relative to $U_1 \times U_2$.
\end{theorem}
\begin{proof}
For $j=1,2$, let $\{P_{j,0},\dots,P_{j,t-1}\}$ be a $(c_j,t)$ LP-packing in $G_j$ relative to~$U_j$ and define
\[
K_\ell = P_{1,\ell} U_2 + U_1 P_{2,\ell} + \sum_{i=0}^{t-1} P_{1,i}P_{2,i+\ell} \quad \mbox{for $\ell = 0,\dots,t-1$},
\]
where the subscript $i+\ell$ is reduced modulo~$t$. We shall use Lemma~\ref{lem-LPpacking} to show that $\{K_0,K_1,\dots,$ $K_{t-1}\}$ is a $(tc_1c_2,t)$ LP-packing in $G_1 \times G_2$ relative to $U_1 \times U_2$.

Each $P_{j,i}$ is a $c_j(tc_j-1)$-subset of $G_j-U_j$, so each $K_\ell$ is a subset of $G_1 \times G_2$ of size $c_1(tc_1-1)tc_2+tc_1 c_2(tc_2-1) + \sum_{i=0}^{t-1} c_1(tc_1-1)c_2(tc_2-1) = tc_1c_2(t^2c_1c_2-1)$. Let $\chi$ be a nonprincipal character of $G_1 \times G_2$, and let $\chi_j=\chi|_{G_j}$ for $j=1,2$. Then
\[
\chi(K_\ell) = \chi_1(P_{1,\ell}) \chi_2(U_2) + \chi_1(U_1) \chi_2(P_{2,\ell}) + \sum_{i=0}^{t-1} \chi_1(P_{1,i})\chi_2(P_{2,i+\ell}) \quad \mbox{for $\ell = 0,\dots,t-1$}.
\]
By Lemma~\ref{lem-LPpacking}, it remains to prove that
\begin{equation}\label{eqn-LPprod}
\{\chi(K_0),\dots,\chi(K_{t-1})\} = \begin{cases}
  \{-tc_1c_2,\dots, -tc_1c_2\} 	    		& \mbox{if $\chi \in (U_1 \times U_2)^\perp$,} \\
  \{t(t-1)c_1c_2, -tc_1c_2, \dots, -tc_1c_2\} 	& \mbox{if $\chi \notin (U_1 \times U_2)^\perp$.}
\end{cases}
\end{equation}

For $j=1,2$, by Lemma~\ref{lem-LPpacking} we are given that
\[
\{\chi_j(P_{j,0}),\dots,\chi_j(P_{j,t-1})\} = \begin{cases}
  \{c_j(tc_j-1),\dots, c_j(tc_j-1)\} 	& \mbox{if $\chi_j \in G_j^\perp$,} \\
  \{-c_j,\dots, -c_j\} 	    		& \mbox{if $\chi_j \in U_j^\perp \sm G_j^\perp$,} \\
  \{(t-1)c_j, -c_j, \dots, -c_j\} 	& \mbox{if $\chi_j \notin U_j^\perp$.}
\end{cases}
\]
Since $\chi$ is nonprincipal on $G_1 \times G_2$, we may assume by symmetry that $\chi_2$ is nonprincipal on $G_2$. Again by symmetry, we need to consider only the following three cases, which together establish~\eqref{eqn-LPprod}.
\begin{description}
\item[{\bf Case 1:}] $\chi_1$ is principal on $U_1$ and $\chi_2$ is principal on $U_2$. \\
Then $\chi_1(P_{1,i})$ is constant over $i$ (regardless of whether $\chi_1$ is principal or nonprincipal on $G_1$), so
$\chi(K_\ell) = \chi_1(P_{1,\ell}) tc_2 + tc_1 (-c_2) + \chi_1(P_{1,\ell}) \sum_{i=0}^{t-1} (-c_2) = -tc_1c_2$ for each~$\ell$.

\item[{\bf Case 2:}] $\chi_1$ is principal on $U_1$ and $\chi_2$ is nonprincipal on $U_2$. \\
Then $\chi_1(P_{1,i})$ is again constant over $i$ and $\sum_{i=0}^{t-1} \chi_2(P_{2,i}) = 0$, so
$\chi(K_\ell) = tc_1 \chi_2(P_{2,\ell})$ which gives
$\{\chi(K_0),\dots,\chi(K_{t-1})\} = \{t(t-1)c_1c_2,-tc_1c_2,\dots,-tc_1c_2\}$.

\item[{\bf Case 3:}] $\chi_1$ is nonprincipal on $U_1$ and $\chi_2$ is nonprincipal on $U_2$. \\
Then
$\chi(K_\ell) = \sum_{i=0}^{t-1} \chi_1(P_{1,i}) \chi_2(P_{2,i+\ell})$,
so the multiset $\{\chi(K_0),\dots,\chi(K_{t-1})\}$ contains
one occurrence of the value $(t-1)c_1(t-1)c_2 + (t-1)(-c_1)(-c_2) = t(t-1)c_1c_2$, and
$t-1$ occurrences of the value $(t-1)c_1(-c_2) + (-c_1)(t-1)c_2 + (t-2)(-c_1)(-c_2) = -tc_1c_2$.
\end{description}
\end{proof}

\begin{remark}
The construction of Theorem~\ref{thm-LPpackingproduct} is closely modelled on the product construction for collections of Latin square type PDSs presented by Polhill in {\rm \cite[Thm.~2.2]{P08DCC}}.
Indeed, the configuration described in {\rm \cite[Lemma 2.1]{P08DCC}} for $e=1$ can be represented as $\{P_1,\dots,P_{t-1},P_t+U-1_G\}$, where $\{ P_1, \dots, P_t\}$ is a $(c,t)$ LP-packing in $G$ relative to~$U$ (see also Remark~\ref{rem-LPpacking}~$(i)$).
Our construction has the advantage that, by identifying the role of the avoided subgroup $U$ in the definition of an LP-packing, we are able to control the avoided subgroup $U_1 \times U_2$ in Theorem~\ref{thm-LPpackingproduct}.
\end{remark}

We extend the construction of Theorem~\ref{thm-LPpackingproduct} using Lemma~\ref{lem-LPpackingunion}.
\begin{corollary}\label{coro-LPpackingproduct}
For $j=1,2$, suppose there exists a $(c_j,t_j)$ LP-packing in an abelian group $G_j$ of order $t_j^2c_j^2$ relative to a subgroup $U_j$ of order~$t_jc_j$, and suppose that $t_1$ divides~$t_2$. Then there exists a $(t_2c_1c_2,t_1)$ LP-packing in $G_1 \times G_2$ relative to $U_1 \times U_2$.
\end{corollary}
\begin{proof}
By Lemma~\ref{lem-LPpackingunion} with $s=\frac{t_2}{t_1}$, there exists a $(\frac{t_2}{t_1}c_2,t_1)$ LP-packing in $G_2$ relative to $U_2$. Then by Theorem~\ref{thm-LPpackingproduct}, there exists a $(t_2c_1c_2,t_1)$ LP-packing in $G_1 \times G_2$ relative to $U_1 \times U_2$.
\end{proof}

\section{LP-partitions}\label{sec4}

In this section, we introduce a $(c,t)$ LP-partition as an auxiliary configuration in the construction of an LP-packing in a larger group from an LP-packing in a smaller group. We then show how to construct an LP-partition from a collection of LP-partitions in various factor groups. We shall apply these two constructions recursively in Section~\ref{sec5} to produce infinite families of LP-packings.

\begin{definition}[LP-partition]\label{def-LPpartition}
Let $t>1$ and $c>0$ be integers. Let $G$ be an abelian group of order~$t^2c^2$, let $V$ be a subgroup of $G$ of order~$tc^2$, and let $H \leqslant V$.
A \emph{$(c,t)$ LP-partition in $G-V$ relative to $H$} is a collection $\{R_1, \ldots, R_t\}$ of $t$ disjoint $(t-1)c^2$-subsets of $G$ whose union is $G-V$ and for which the multiset equality
\[
\{\chi(R_1), \dots, \chi(R_t)\}= \begin{cases}
 \{-c^2,\dots, -c^2\} 	& \mbox{if $\chi \in V^{\perp}$,} \\
 \{0,\dots,0\} 		& \mbox{if $\chi \in H^{\perp} \sm V^{\perp}$,} \\
 \{(t-1)c,-c,\dots,-c\} & \mbox{if $\chi \notin H^{\perp}$}.
\end{cases}
\]
holds for all nonprincipal characters $\chi$ of~$G$.
\end{definition}

\begin{remark}\label{rem-LPpartitiondefn}
\quad
\begin{enumerate}[$(i)$]
\item
In Definition~\ref{def-LPpartition}, we can deduce that the $R_i$ are disjoint and their union is $G-V$ from the other conditions, by applying the Fourier inversion formula to the equation
\[
\chi \Big(\sum_{i=1}^t R_i\Big) = \begin{cases}
  t(t-1)c^2 	& \mbox{if $\chi \in G^\perp$}, \\
  -tc^2       	& \mbox{if $\chi \in V^\perp \sm G^\perp$}, \\
  0         	& \mbox{if $\chi \notin V^\perp$}.
\end{cases}
\]

\item
Each subset of a $(c,t)$ LP-partition in $G-V$ relative to $H$ is a $(G,V,H;c)$ Latin shell, as introduced in {\rm\cite[Defn.~3.1]{HN}} as an auxiliary configuration in the construction of a single Latin square type PDS.
\end{enumerate}
\end{remark}

A simple example of an LP-partition is given by a spread of a vector space (see Definition~\ref{def-spread}).

\begin{example}\label{ex-spread2}
By Lemma~\ref{lem-LPpacking} and Definition~\ref{def-LPpartition}, a $(1,t)$ LP-packing in $G$ relative to~$U$ is identical to a $(1,t)$ LP-partition in $G-U$ relative to~$U$.
Let $t=p^s$ for a prime $p$ and positive integer~$s$, and let $V$ be a $2s$-dimensional vector space over~$\F_p$ with identity~$1_V$.
Let $\{H_0,\dots,H_{p^s}\}$ be a spread of $V \cong \Z_p^{2s}$.
Following Example~\ref{ex-spread}, $\{H_1-1_V,\dots,H_{p^s}-1_V\}$ is therefore a $(1,p^s)$ LP-partition in $V-H_0$ relative to~$H_0$ as well as a $(1,p^s)$ LP-packing in $V$ relative to~$H_0$.
\end{example}

We next combine the subsets of an LP-partition with the subsets obtained by lifting an LP-packing, in order to produce an LP-packing in a larger group.

\begin{theorem}\label{thm-composition}
Let $G$ be an abelian group of order $t^4c^2$ containing subgroups $H \leqslant U \leqslant V \leqslant G$, where $H$, $U$, $V$ have order $t$, $t^2c$, $t^3c^2$, respectively.
Suppose there exists a $(tc,t)$ LP-partition in $G-V$ relative to $H$, and a $(c,t)$ LP-packing in $V/H$ relative to~$U/H$.
Then there exists a $(tc,t)$ LP-packing in $G$ relative to~$U$.
\end{theorem}
\begin{proof}
Let $\{R_1,\dots,R_t\}$ be a $(tc,t)$ LP-partition in $G-V$ relative to $H$, and let $\{P_1,\dots,P_t\}$ be a $(c,t)$ LP-packing in $V/H$ relative to~$U/H$.
For $i=1,\dots,t$, let $P'_i = \{g \in V \mid gH \in P_i\}$ be the pre-image of $P_i$ under the quotient mapping from $V$ to $V/H$. We shall use Lemma~\ref{lem-LPpacking} to show that $\{P'_1+R_1,\dots,P'_t+R_t\}$ is a $(tc,t)$ LP-packing in $G$ relative to~$U$.

Each $P_i$ is a $c(tc-1)$-subset of $V/H-U/H$ by Definition~\ref{def-LPpacking}, so
each $P'_i$ is a $tc(tc-1)$-subset of~$V-U$.
Since each $R_i$ is a $(t-1)(tc)^2$-subset of $G-V$ by Definition~\ref{def-LPpartition}, it follows that each $P'_i+R_i$ is a subset of $G$ of size
$|P'_i| + |R_i| = tc(t^2c-1)$.

By Lemma~\ref{lem-LPpacking}, for all characters $\psi$ of $V/H$ we have
\[
\{\psi(P_1),\dots,\psi(P_t)\}=
  \begin{cases}
    \{c(tc-1),\dots,c(tc-1)\} 	& \mbox{if $\psi \in (V/H)^{\perp}$,} \\
    \{-c,\dots, -c\} 		& \mbox{if $\psi \in (U/H)^{\perp} \sm (V/H)^{\perp}$,} \\
    \{(t-1)c, -c, \dots, -c\} 	& \mbox{if $\psi \notin (U/H)^{\perp}$.}
  \end{cases}
\]
Let $\chi$ be a nonprincipal character of~$G$. If $\chi$ is principal on $H$ then it induces a nonprincipal character $\psi$ on~$G/H$. Since $P'_i$ is a union of cosets of the order~$t$ subgroup $H$, we obtain
\[
\chi(P'_i) = \begin{cases}
 t\psi(P_i)	& \mbox{if $\chi \in H^\perp$}, \\
 0 		& \mbox{if $\chi \notin H^\perp$}.
\end{cases}
\]
Therefore the value of $\{\chi(P'_1),\dots,\chi(P'_t)\}$ is as specified in the second column of the following table:

\[
\ra{1.5}
\scalebox{0.85}{
$
\begin{array}{|l|c|c|c|}
\cline{2-4}
\mc{1}{l|}{}			& \{\chi(P'_1),\dots,\chi(P'_t)\}	& \{\chi(R_1),\dots,\chi(R_t)\} 	& \{\chi(P'_1+R_1),\dots,\chi(P'_t+R_t)\} 	\\ \hline
\chi \in V^\perp \sm G^\perp 	& \{tc(tc-1),\dots,tc(tc-1)\}		& \{-t^2c^2,\dots,-t^2c^2\} 		& \mr{2}{*}{$\{-tc,\dots,-tc\}$} 		\\ \cline{1-3}
\chi \in U^\perp \sm V^\perp 	& \{-tc,\dots,-tc\}			& \mr{2}{*}{$\{0,\dots,0\}$}		& 						\\ \cline{1-2} \cline{4-4}
\chi \in H^\perp \sm U^\perp 	& \{t(t-1)c,-tc,\dots,-tc\}		& 					& \mr{2}{*}{$\{t(t-1)c,-tc,\dots,-tc\}$} 	\\ \cline{1-3}
\chi \notin H^\perp 		& \{0,\dots,0\}				& \{t(t-1)c,-tc,\dots,-tc\} 		& 						\\ \hline
\end{array}
$
}
\]

By Definition~\ref{def-LPpartition}, the value of $\{\chi(R_1),\dots,\chi(R_t)\}$ is as given in the third column of the table, and the fourth column contains the sum of the second and third columns. Lemma~\ref{lem-LPpacking} then implies that
$\{P'_1+R_1,\dots,P'_t+R_t\}$ is a $(tc,t)$ LP-packing in $G$ relative to~$U$, as required.

\end{proof}

\begin{remark}\label{rem-lifting}
Lifting constructions similar to that of Theorem~\ref{thm-composition} were proposed in {\rm \cite{H00}, \cite[Sects.~4,~5]{H03}}, {\rm\cite{HN}}, each producing a single Latin square type partial difference set. These previous constructions make use of the delicate structure of finite local rings. In contrast, the construction of Theorem~\ref{thm-composition} applies simultaneously to a collection of such partial difference sets, and does not require ring theory.
\end{remark}

We now construct an LP-partition by lifting and combining the subsets of a collection of LP-partitions in various factor groups. The construction makes use of two spreads of a vector space over a finite field (as introduced in Definition~\ref{def-spread}).

\begin{theorem}\label{thm-LPpartitionrecursive}
Let $t=p^s$ for a prime $p$ and positive integer~$s$.
Let $G$ be an abelian group of order $t^4c^2$ containing subgroups $Q \leqslant G' \leqslant G$, where $Q \cong \Z_p^{2s}$ and $G/G' \cong \Z_p^{2s}$.
Let $H_0,\dots,H_t$ be subgroups of $G$ forming a spread when viewed as subgroups of~$Q$.
Let $V_0,\dots,V_t$ be subgroups of $G$ for which $\{V_0/G',\dots,V_t/G'\}$ is a spread of~$G/G'$, and for which $H_i \leqslant V_i$ for each $i=1,\dots,t$.
Suppose that for each $i = 1,\dots,t$ there exists a $(c,t)$ LP-partition in $V_i/H_i - G'/H_i$ relative to $Q/H_i$.
Then there exists a $(tc,t)$ LP-partition in $G-V_0$ relative to~$H_0$.
\end{theorem}
\begin{proof}
For each $i = 1, \dots, t$, let $\{ S_{i1},\dots,S_{it}\}$ be a $(c,t)$ LP-partition in $V_i/H_i-G'/H_i$ relative to~$Q/H_i$. By Definition~\ref{def-LPpartition}, for each $i=1,\dots,t$ and for all nonprincipal characters $\psi_i$ of $V_i/H_i$ we have
\[
\{\psi_i(S_{i1}),\dots,\psi_i(S_{it})\}= \begin{cases}
 \{-c^2,\dots, -c^2\} 	   & \mbox{if $\psi_i \in (G'/H_i)^{\perp}$}, \\
 \{0,\dots,0\}        	   & \mbox{if $\psi_i \in (Q/H_i)^{\perp} \sm (G'/H_i)^{\perp}$,} \\
 \{(t-1)c, -c, \dots, -c\} & \mbox{if $\psi_i \notin (Q/H_i)^{\perp}$.}
  \end{cases}
\]
For each $i,j$ satisfying $1 \le i, j \le t$, let
$S'_{ij}=\{g \in V_i \mid gH_i \in S_{ij}\}$
be the pre-image of $S_{ij}$ under the quotient mapping from $V_i$ to~$V_i/H_i$, and let $R_j=\sum_{i=1}^t S'_{ij}$ for $j = 1,\dots,t$.
We shall show that $\{R_1,\dots,R_t\}$ is a $(tc,t)$ LP-partition in $G-V_0$ relative to~$H_0$.

By Definition~\ref{def-LPpartition},
each $S_{ij}$ is a $(t-1)c^2$-subset of $V_i/H_i-G'/H_i$, and so
each $S'_{ij}$ is a $t(t-1)c^2$-subset of~$V_i-G'$.
By Example~\ref{ex-spread}, the $V_i/G'-1_{G/G'}$ are disjoint in $G/G'$ and so
the $V_i-G'$ are disjoint in~$G$.
Therefore each $R_j = \sum_{i=1}^t S'_{ij}$ is a $t^2(t-1)c^2$-subset of~$G$.
By Definition~\ref{def-LPpartition} and Remark~\ref{rem-LPpartitiondefn}~$(i)$, it is now sufficient to show that, for all nonprincipal characters $\chi$ of~$G$,
\begin{equation}\label{eqn-Richars}
\{\chi(R_1),\dots,\chi(R_t)\} = \begin{cases}
  \{-t^2c^2,\dots, -t^2c^2\}   & \mbox{if $\chi \in V_0^\perp$}, \\
  \{0,\dots,0\} 	       & \mbox{if $\chi \in H_0^\perp \sm V_0^\perp$},\\
  \{t(t-1)c, -tc, \dots, -tc\} & \mbox{if $\chi \notin H_0^\perp$}.
\end{cases}
\end{equation}

Let $\chi$ be a nonprincipal character of~$G$.
Since the $H_i$ correspond to a spread of $Q$, if $\chi$ is nonprincipal on $Q$, then $\chi$ is principal on exactly one of the~$H_i$ (see Example~\ref{ex-spread}). If $\chi$ is nonprincipal on $V_i$ and principal on $H_i$, then $\chi$ induces a nonprincipal character $\psi_i$ on~$V_i/H_i$; since each $S'_{ij}$ is a union of cosets of the order $t$ subgroup $H_i$, we have for each $i,j$ satisfying $1 \le i,j \le t$ that
\[
\chi(S'_{ij}) = \begin{cases}
  t \psi_i(S_{ij}) & \mbox{if $\chi \in H_i^\perp \sm V_i^\perp$}, \\
  0		   & \mbox{if $\chi \notin H_i^\perp$}.
\end{cases}
\]
If $\chi$ is principal on $G'$ then it induces a nonprincipal character~$\tau$ on~$G/G'$;
since the $V_i/G'$ form a spread of $G/G'$, we then have that $\tau$ is principal on exactly one of the~$V_i/G'$ and so $\chi$ is principal on exactly one of the~$V_i$.
We shall use the subgroup inclusions
\begin{equation}\label{eqn-inclusions}
H_i \leqslant Q \leqslant G' \leqslant V_i \quad \mbox{for each $i = 0,\dots,t$}.
\end{equation}

The conclusions of the following five cases collectively establish~\eqref{eqn-Richars}.

\begin{description}

\item[{\bf Case 1:}] $\chi$ is principal on $V_0$. \\
Then by \eqref{eqn-inclusions}, $\chi$ is principal on $G'$ and on each~$H_i$.
For each $i=1,\dots,t$, therefore $\chi$ is nonprincipal on $V_i$ and principal on~$G'/H_i$.
For each $j$, we obtain
$\chi(R_j) = \sum_{i=1}^t \chi(S'_{ij}) = t \sum_{i=1}^t \psi_i(S_{ij}) = t \sum_{i=1}^t (-c^2) = -t^2c^2$.

\item[{\bf Case 2:}] $\chi$ is principal on $G'$ and nonprincipal on $V_0$. \\
Then by \eqref{eqn-inclusions}, $\chi$ is principal on each~$H_i$.
Therefore $\chi$ is principal on each $G'/H_i$ and on exactly one of $V_1,\dots,V_t$, say~$V_I$.
For each $j$, we obtain
$\chi(R_j) = \sum_{i=1}^t \chi(S'_{ij}) = |S'_{Ij}| + t \sum_{1 \le i \le t,\, i\ne I} \psi_i(S_{ij}) = t(t-1)c^2 + t \sum_{1 \le i \le t,\, i \ne I} (-c^2) = 0$.

\item[{\bf Case 3:}] $\chi$ is principal on $Q$ and nonprincipal on $G'$. \\
Then by \eqref{eqn-inclusions}, $\chi$ is principal on each~$H_i$ and nonprincipal on each~$V_i$.
Therefore $\chi$ is principal on each $Q/H_i$ and nonprincipal on each~$G'/H_i$.
For each $j$, we obtain
$\chi(R_j) = t \sum_{i=1}^t \psi_i(S_{ij}) = 0$.

\item[{\bf Case 4:}] $\chi$ is principal on $H_0$ and nonprincipal on $Q$. \\
Then for $i=1,\dots,t$, we have that $\chi$ is nonprincipal on $H_i$ and so $\chi(S'_{ij}) = 0$ for each~$j$. Therefore $\chi(R_j)=0$ for each~$j$.

\item[{\bf Case 5:}] $\chi$ is nonprincipal on $H_0$.\\
By \eqref{eqn-inclusions}, $\chi$ is nonprincipal on $Q$ and on each $V_i$.
Therefore $\chi$ is principal on exactly one of $H_1,\dots,H_t$, say $H_I$, and is nonprincipal on~$Q/H_I$.
For each $j$, we obtain
$\chi(R_j) = t \psi_I(S_{Ij})$ and so
$\{\chi(R_1),\dots,\chi(R_t)\} = \{t(t-1)c, -tc, \ldots, -tc\}$.

\end{description}

\end{proof}

The following example illustrates the use of Theorems~\ref{thm-composition} and~\ref{thm-LPpartitionrecursive} to construct a $(4,4)$ LP-packing in $\Z_4^4$ relative to each of the order~$16$ subgroups $\Z_4\times(2\Z_4)^2$ and $\Z_4^2$ and $(2\Z_4)^4$ in turn, starting from a spread of~$\F_4^2$. In Section~\ref{sec5} we shall show how to apply this procedure recursively to produce infinite families of LP-packings and LP-partitions in groups of increasing exponent and fixed rank.
\begin{example}
Let $(\F_4,\oplus,\cdot) = \{0,1,\alpha,\alpha^2\}$, where $\alpha^2=\alpha \oplus 1$.
Then a spread of $(\F_4^2,\oplus)$ is $\big\{\lan (1,0) \ran$, $\lan (0,1) \ran$, $\lan (1,1) \ran$, $\lan (\alpha,1) \ran$, $\lan (\alpha \oplus 1,1) \ran \big\}$.
Using the isomorphism from $(\F_4^2,\oplus)$ to $\Z_2^4 = \lan y_1,y_2,y_3,y_4 \ran$ given by
$(1,0) \mapsto y_1$, $(\alpha,0) \mapsto y_2$, $(0,1) \mapsto y_3$, $(0,\alpha) \mapsto y_4$,
this spread in multiplicative notation is
\begin{equation}\label{eqn-spreadZ24}
\big\{ \lan y_1, y_2 \ran, \lan y_3, y_4 \ran, \lan y_1y_3, y_2y_4 \ran, \lan y_2y_3,y_1y_2y_4 \ran, \lan y_1y_2y_3, y_1y_4 \ran \big\}.
\end{equation}
Let
\[
\begin{split}
S_1(y_1,y_2,y_3,y_4) &= y_3+y_4+y_3y_4,		\\
S_2(y_1,y_2,y_3,y_4) &= y_1y_3+y_2y_4+y_1y_2y_3y_4,	\\
S_3(y_1,y_2,y_3,y_4) &= y_2y_3+y_1y_2y_4+y_1y_3y_4,	\\
S_4(y_1,y_2,y_3,y_4) &= y_1y_2y_3+y_1y_4+y_2y_3y_4.
\end{split}
\]
Then by Examples~\ref{ex-spread} and~\ref{ex-spread2},
\begin{equation}\label{eqn-packingpartition}
\big\{ S_1(y_1,y_2,y_3,y_4), S_2(y_1,y_2,y_3,y_4), S_3(y_1,y_2,y_3,y_4), S_4(y_1,y_2,y_3,y_4) \big\}
\end{equation}
is both a $(1,4)$ LP-partition in $\lan y_1,y_2,y_3,y_4 \ran -\lan y_1, y_2 \ran \cong \Z_2^4 - \Z_2^2$ relative to $\lan y_1, y_2 \ran \cong \Z_2^2$,
and a $(1,4)$ LP-packing in $\lan y_1,y_2,y_3,y_4 \ran = \Z_2^4$ relative to $\lan y_1, y_2 \ran \cong \Z_2^2$.

Let $G=\Z_4^4 = \lan x_1,x_2,x_3,x_4 \ran$.
We firstly use Theorem~\ref{thm-LPpartitionrecursive} to construct a $(4,4)$ LP-partition in $G-\lan x_1,x_2,x_3^2,x_4^2 \ran = \Z_4^4 - \Z_4^2\times(2\Z_4)^2$
relative to~$\lan x_1^2,x_2^2 \ran \cong \Z_2^2$.
Let $Q = \lan x_1^2,x_2^2,x_3^2,x_4^2 \ran \cong \Z_2^4$,
and let $G' = \lan x_1^2,x_2^2,x_3^2,x_4^2 \ran$ so that $G/G' \cong \Z_2^4$.
With reference to \eqref{eqn-spreadZ24}, the subgroups
\[
\begin{split}
& H_0 = \lan x_1^2, x_2^2 \ran, \quad H_1 = \lan x_3^2, x_4^2 \ran, \quad H_2 = \lan x_1^2x_3^2, x_2^2x_4^2 \ran, \quad \\
& H_3 = \lan x_2^2x_3^2,x_1^2x_2^2x_4^2 \ran, \quad H_4 = \lan x_1^2x_2^2x_3^2, x_1^2x_4^2 \ran
\end{split}
\]
form a spread when viewed as subgroups of~$Q$.
Let
\[
\begin{split}
& V_0 = \lan x_1,x_2,x_3^2,x_4^2 \ran, \quad
  V_1 = \lan x_1^2,x_2^2,x_3,x_4 \ran, \quad
  V_2 = \lan x_1^2,x_2^2,x_1x_3,x_2x_4 \ran, \quad \\
& V_3 = \lan x_1^2,x_2^2,x_2x_3,x_1x_2x_4 \ran, \quad
  V_4 = \lan x_1^2,x_2^2,x_1x_2x_3,x_1x_4 \ran,
\end{split}
\]
so that, with reference to \eqref{eqn-spreadZ24},
$\{V_0/G',V_1/G',V_2/G',V_3/G',V_4/G'\}$ is a spread of $G/G'$ and $H_i \leqslant V_i$ for $i=1,2,3,4$.
Then the factor groups $V_i/H_i, G'/H_i, Q/H_i$ take the following form for $i=1,2,3,4$:
\[
\ra{1.2}
\begin{array}{r|c|c|c}
i	& V_i/H_i								& G'/H_i					& Q/H_i						\\ \hline
1	& \lan x_1^2H_1, x_2^2H_1, x_3H_1, x_4H_1 \ran \cong \Z_2^4  		& \lan x_1^2H_1, x_2^2H_1 \ran \cong \Z_2^2  	& \lan x_1^2H_1, x_2^2H_1 \ran \cong \Z_2^2 	\\
2	& \lan x_1^2H_2, x_2^2H_2, x_1x_3H_2, x_2x_4H_2 \ran \cong \Z_2^4 	& \lan x_1^2H_2, x_2^2H_2 \ran \cong \Z_2^2  	& \lan x_1^2H_2, x_2^2H_2 \ran \cong \Z_2^2 	\\
3	& \lan x_1^2H_3, x_2^2H_3, x_2x_3H_3, x_1x_2x_4H_3 \ran \cong \Z_2^4  	& \lan x_1^2H_3, x_2^2H_3 \ran \cong \Z_2^2   	& \lan x_1^2H_3, x_2^2H_3 \ran \cong \Z_2^2 	\\
4	& \lan x_1^2H_4, x_2^2H_4, x_1x_2x_3H_4, x_1x_4H_4 \ran \cong \Z_2^4 	& \lan x_1^2H_4, x_2^2H_4 \ran \cong \Z_2^2  	& \lan x_1^2H_4, x_2^2H_4 \ran \cong \Z_2^2
\end{array}
\]

Since \eqref{eqn-packingpartition} is a $(1,4)$ LP-partition in $\lan y_1,y_2,y_3,y_4 \ran - \lan y_1,y_2 \ran$ relative to $\lan y_1,y_2 \ran$,
where $\lan y_1,y_2,y_3,y_4\ran = \Z_2^4$, it follows for $i=1,2,3,4$ that the $4$ subsets of $V_i/H_i$
shown in row $i$ of the following table form a $(1,4)$ LP-partition in $V_i/H_i-G'/H_i$ relative to $Q/H_i$:
\begin{small}
\[
\ra{1.2}
\scalemath{0.83}{
\begin{array}{|l|l|l|l|}
																				   \hline
S_1(x_1^2,x_2^2,x_3,x_4)H_1		& S_2(x_1^2,x_2^2,x_3,x_4)H_1		& S_3(x_1^2,x_2^2,x_3,x_4)H_1		& S_4(x_1^2,x_2^2,x_3,x_4)H_1		\\
S_1(x_1^2,x_2^2,x_1x_3,x_2x_4)H_2	& S_2(x_1^2,x_2^2,x_1x_3,x_2x_4)H_2	& S_3(x_1^2,x_2^2,x_1x_3,x_2x_4)H_2	& S_4(x_1^2,x_2^2,x_1x_3,x_2x_4)H_2	\\
S_1(x_1^2,x_2^2,x_2x_3,x_1x_2x_4)H_3	& S_2(x_1^2,x_2^2,x_2x_3,x_1x_2x_4)H_3	& S_3(x_1^2,x_2^2,x_2x_3,x_1x_2x_4)H_3	& S_4(x_1^2,x_2^2,x_2x_3,x_1x_2x_4)H_3	\\
S_1(x_1^2,x_2^2,x_1x_2x_3,x_1x_4)H_4	& S_2(x_1^2,x_2^2,x_1x_2x_3,x_1x_4)H_4	& S_3(x_1^2,x_2^2,x_1x_2x_3,x_1x_4)H_4	& S_4(x_1^2,x_2^2,x_1x_2x_3,x_1x_4)H_4	\\ \hline
\end{array}
}
\]
\end{small}
Now regard each entry of the above table as a $12$-subset of $G$, by interpreting $H_i$ as the sum of four elements in~$G$.
By Theorem~\ref{thm-LPpartitionrecursive}, summing the entries in each column of the table then gives a $(4,4)$ LP-partition $\{R_1,R_2,R_3,R_4\}$ in $G-V_0$ relative to~$H_0$.

We next use Theorem~\ref{thm-composition} to construct a $(4,4)$ LP-packing in $G = \Z_4^4$ relative to $U = \lan x_1,x_2^2,x_3^2 \ran \cong \Z_4 \times (2\Z_4)^2$.
We have $V_0/H_0 = \lan x_1H_0, x_2H_0, x_3^2H_0, x_4^2H_0 \ran \cong \Z_2^4$ and $U/H_0 = \lan x_1H_0, x_3^2H_0 \ran \cong \Z_2^2$.
Since \eqref{eqn-packingpartition} is a $(1,4)$ LP-packing in $\lan y_1,y_2,y_3,y_4 \ran = \Z_2^4$ relative to $\lan y_1,y_2 \ran$,
it follows that
\[
\big\{
S_1(x_1,x_3^2,x_2,x_4^2)H_0,\,\,
S_2(x_1,x_3^2,x_2,x_4^2)H_0,\,\,
S_3(x_1,x_3^2,x_2,x_4^2)H_0,\,\,
S_4(x_1,x_3^2,x_2,x_4^2)H_0
\big\}
\]
is a $(1,4)$ LP-packing in $V_0/H_0$ relative to~$U/H_0$.
Regard each of these subsets as a $12$-subset of $G$, by interpreting $H_0$ as the sum of four elements in~$G$.
Then by Theorem~\ref{thm-composition},
\[
\begin{split}
\big\{
R_1+ S_1(x_1,x_3^2,x_2,x_4^2)H_0,\,\,
  R_2+ S_2(x_1,x_3^2,x_2,x_4^2)H_0,\,\, \\
  R_3+ S_3(x_1,x_3^2,x_2,x_4^2)H_0,\,\,
  R_4+ S_4(x_1,x_3^2,x_2,x_4^2)H_0
\big\}
\end{split}
\]
is a $(4,4)$ LP-packing in $G$ relative to~$U$.

Using the same $(4,4)$ LP-partition $\{R_1,R_2,R_3,R_4\}$, we can similarly use
Theorem~\ref{thm-composition} to construct a $(4,4)$ LP-packing in $G$ relative to the subgroup $U' = \lan x_1,x_2 \ran \cong \Z_4^2$ and to $U'' = \lan x_1^2,x_2^2,x_3^2,x_4^2 \ran \cong (2\Z_4)^4$.
Following the above procedure, and using
$U'/H_0 = \lan x_1H_0, x_2H_0 \ran \cong \Z_2^2$ and
$U''/H_0 = \lan x_3^2H_0, x_4^2H_0 \ran \cong \Z_2^2$, shows that
\[
\begin{split}
\big\{
  R_1+ S_1(x_1,x_2,x_3^2,x_4^2)H_0,\,\,
  R_2+ S_2(x_1,x_2,x_3^2,x_4^2)H_0,\,\, \\
  R_3+ S_3(x_1,x_2,x_3^2,x_4^2)H_0,\,\,
  R_4+ S_4(x_1,x_2,x_3^2,x_4^2)H_0
\big\}
\end{split}
\]
is a $(4,4)$ LP-packing in $G$ relative to~$U'$, and that
\[
\begin{split}
\big\{
  R_1+ S_1(x_3^2,x_4^2,x_1,x_2)H_0,\,\,
  R_2+ S_2(x_3^2,x_4^2,x_1,x_2)H_0,\,\, \\
  R_3+ S_3(x_3^2,x_4^2,x_1,x_2)H_0,\,\,
  R_4+ S_4(x_3^2,x_4^2,x_1,x_2)H_0
\big\}
\end{split}
\]
is a $(4,4)$ LP-packing in $G$ relative to~$U''$.
\end{example}

\section{Recursive Construction of LP-partitions and LP-packings}\label{sec5}

In this section, we recursively construct infinite families of LP-partitions and LP-packings, as shown schematically in Figure~\ref{fig-ladder}.
On the right side of the figure, we iteratively use Theorem~\ref{thm-LPpartitionrecursive} to produce an LP-partition in the group $\Z_{p^a}^{2s} - \Z_{p^a}^s \times (p\Z_{p^a})^s$ relative to $H_0 \cong \Z_p^s$ by lifting with respect to $H_i \cong \Z_p^s$ an LP-partition in a factor group, for each $i=1,\dots,p^s$ (see Theorem~\ref{thm-LPpartitionuniform}).
On the left side of the figure, we iteratively use Theorem~\ref{thm-composition} to produce an LP-packing in $\Z_{p^a}^{2s}$ relative to an order $p^{as}$ subgroup, by lifting with respect to~$H_0$ an LP-packing in a factor group and taking the union of the resulting subsets with the subsets of the LP-partition at the same height (see Theorem~\ref{thm-LPpackinguniform}).

\begin{figure}[ht]
\centering

\setlength{\unitlength}{1mm}
\begin{picture}(150,130)(10,5)

\put(50,5){\makebox(0,0){$(1,p^s)$ LP-packing in $\Z_p^{2s}$}}
\put(50,0){\makebox(0,0){relative to order $p^s$ subgroup}}
\put(50,10){\vector(0,1){25}}
\put(22,22.5){\makebox(0,0){Lift using $H_0$}}
\put(7,15){\framebox(30,15){}}

\put(125,5){\makebox(0,0){$(1,p^s)$ LP-partition in}}
\put(125,0){\makebox(0,0){$\Z_{p}^{2s}-\Z_p^s$ relative to $H_0 \cong \Z_p^s$}}
\put(120,10){\vector(0,1){25}}
\put(150,25){\makebox(0,0){Lift using}}
\put(150,20){\makebox(0,0){$H_1,\dots,H_{p^s}$}}
\put(135,15){\framebox(30,15){}}

\put(50,45){\makebox(0,0){$(p^s,p^s)$ LP-packing in $\Z_{p^2}^{2s}$}}
\put(50,40){\makebox(0,0){relative to order $p^{2s}$ subgroup}}
\put(50,50){\vector(0,1){25}}
\put(22,62.5){\makebox(0,0){Lift using $H_0$}}
\put(7,55){\framebox(30,15){}}

\put(125,45){\makebox(0,0){$(p^s,p^s)$ LP-partition in}}
\put(125,40){\makebox(0,0){$\Z_{p^2}^{2s}-\Z_{p^2}^s \times (p\Z_{p^2})^s$ relative to $H_0 \cong \Z_p^s$}}
\put(120,50){\vector(0,1){25}}
\put(105,45){\vector(-1,0){30}}
\put(150,65){\makebox(0,0){Lift using}}
\put(150,60){\makebox(0,0){$H_1,\dots,H_{p^s}$}}
\put(135,55){\framebox(30,15){}}

\put(50,85){\makebox(0,0){$(p^{2s},p^s)$ LP-packing in $\Z_{p^3}^{2s}$}}
\put(50,80){\makebox(0,0){relative to order $p^{3s}$ subgroup}}
\put(50,90){\vector(0,1){25}}
\put(22,102.5){\makebox(0,0){Lift using $H_0$}}
\put(7,95){\framebox(30,15){}}

\put(125,85){\makebox(0,0){$(p^{2s},p^s)$ LP-partition in}}
\put(125,80){\makebox(0,0){$\Z_{p^3}^{2s}-\Z_{p^3}^s \times (p\Z_{p^3})^s$ relative to $H_0 \cong \Z_p^s$}}
\put(120,90){\vector(0,1){25}}
\put(105,85){\vector(-1,0){30}}
\put(150,105){\makebox(0,0){Lift using}}
\put(150,100){\makebox(0,0){$H_1,\dots,H_{p^s}$}}
\put(135,95){\framebox(30,15){}}

\put(50,120){\circle*{1}}
\put(50,125){\circle*{1}}

\put(120,120){\circle*{1}}
\put(120,125){\circle*{1}}

\end{picture}

\vspace{1cm}
\caption{Recursive construction of LP-packings and LP-partitions
according to Theorems~\ref{thm-LPpartitionuniform} and~\ref{thm-LPpackinguniform}}
\label{fig-ladder}
\end{figure}
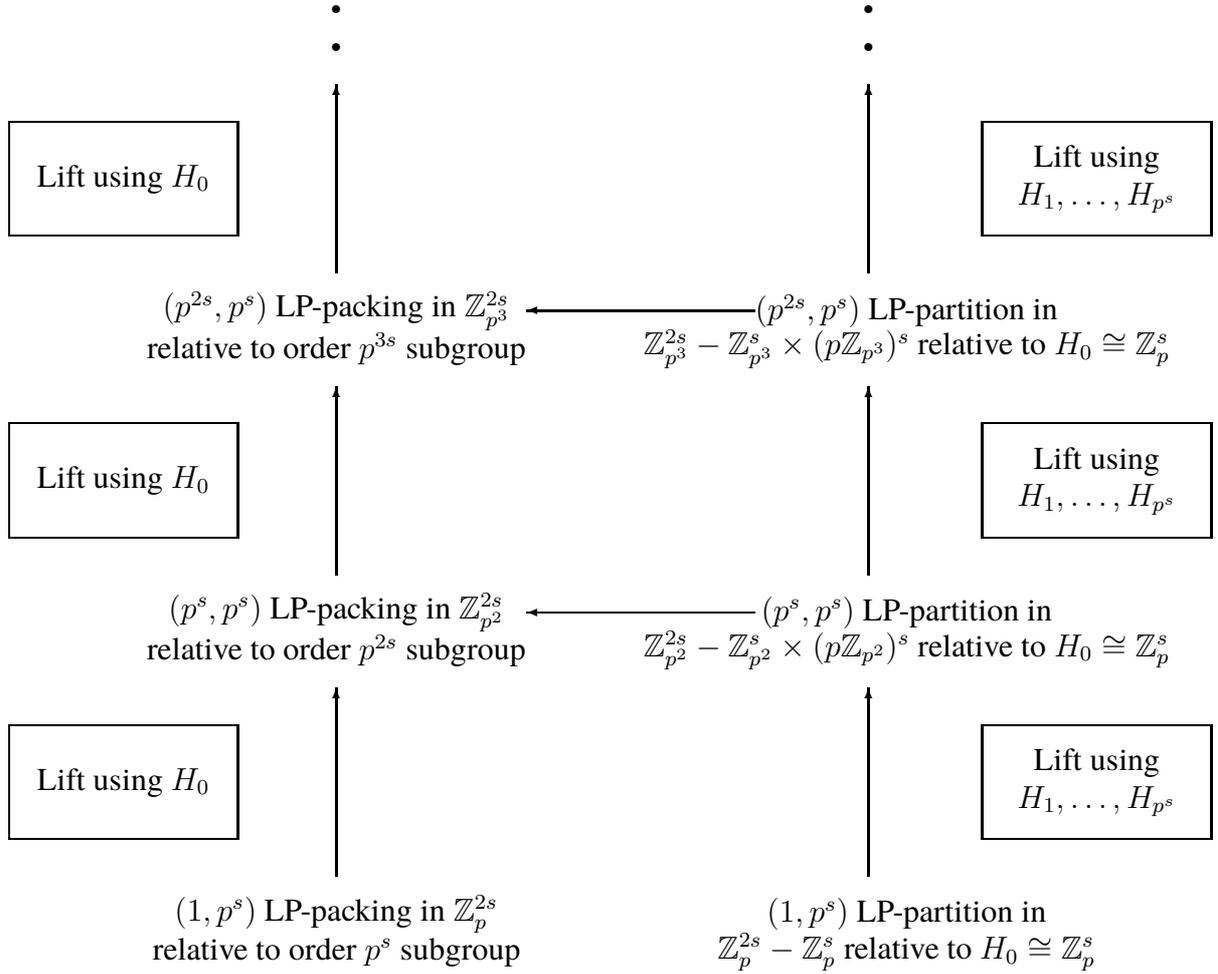
We begin with a technical lemma.

\begin{lemma}\label{lem-quotient}
Let $p$ be prime, and let $s$ and $a_1, \dots, a_{2s}$ be positive integers.
Let $G = \prod_{j=1}^{2s} \Z_{p^{a_j}}$, and let $Q$ be the unique subgroup of $G$ isomorphic to~$\Z_{p}^{2s}$. Let $H_0,\dots,H_{p^s}$ be a spread of~$Q$, where (without loss of generality) $H_0$ is contained in the first $s$ direct factors of~$G$. Then for each $i = 1, \dots, p^s$, we have
$G/H_i \cong \prod_{j=1}^s \Z_{p^{a_j}} \times \prod_{j=s+1}^{2s} \Z_{p^{a_j-1}}$, and the quotient group $Q/H_i$ is contained in the first $s$ direct factors of $G/H_i$.
\end{lemma}
\begin{proof}
Fix $i \in \{1,\dots,p^s\}$.
Since $Q=H_0H_i$, we can choose a set of generators for $Q$ (depending on $i$) so that $H_i$ is contained in the last $s$ direct factors of~$Q$. Therefore $G/H_i \cong \prod_{j=1}^s \Z_{p^{a_j}} \times \prod_{j=s+1}^{2s} \Z_{p^{a_j-1}}$. Let the first $s$ direct factors of~$G$ be $\lan x_1,\dots,x_s \ran$, so that the first $s$ direct factors of $G/H_i$ are $\lan x_1H_i,\dots,x_sH_i \ran$.
Since $H_0 = \lan x_1^{p^{a_1-1}}, \dots, x_s^{p^{a_s-1}} \ran$, we have $Q/H_i=\lan x_1^{p^{a_1-1}}H_i,\dots,x_s^{p^{a_s-1}}H_i \ran$. Therefore $Q/H_i$ is contained in the first $s$ direct factors of~$G/H_i$.
\end{proof}

We now apply Theorem~\ref{thm-LPpartitionrecursive} iteratively to produce an infinite family of LP-partitions.

\begin{theorem}\label{thm-LPpartitionuniform}
Let $p$ be prime, let $a,s$ be positive integers, and let $G=\Z_{p^a}^{2s}$.
Let $H_0$ and $V_0$
be the unique subgroups of $G$ for which $H_0 \cong \Z_p^s$ is contained in the first $s$ direct factors of $G$ and $V_0 \cong \Z_{p^a}^s \times \Z_{p^{a-1}}^s$ satisfies $V_0/H_0 \cong \Z_{p^{a-1}}^{2s}$.
Then there exists a $(p^{(a-1)s},p^s)$ LP-partition in $G-V_0$ relative to~$H_0$.
\end{theorem}
\begin{proof}
The proof is by induction on $a \ge 1$.  The case $a=1$ is given by Example~\ref{ex-spread2}. Suppose that all cases up to $a-1 \ge 1$ are true.
Let $Q$ and $G'$ be the unique subgroups of $G$ for which $Q \cong \Z_p^{2s}$ and $G' \cong \Z_{p^{a-1}}^{2s}$.
Since $G/G' \cong \Z_p^{2s}$, by Example~\ref{ex-spread} there exists a spread in $G/G'$ one of whose elements we may take to be~$V_0/G'$. Let the remaining elements of this spread be $V_1/G',\dots,V_{p^s}/G'$, where $V_1,\dots,V_{p^s}$ are subgroups of~$G$ containing $G'$.
Let $\phi$ be a group isomorphism from $G/G'$ to $Q$ for which $\phi(V_0/G') = H_0$, and let $H_i = \phi(V_i/G')$ for $i = 1,\dots,p^s$.
Since $\{V_0/G',\dots,V_{p^s}/G'\}$ is a spread of $G/G'$, it follows that $\{H_0,\dots,H_{p^s}\}$ is a spread of~$Q$.
Furthermore, for each $i$ we have $H_i \leqslant Q \leqslant G' \leqslant V_i$ and therefore $H_i$ is a subgroup of $V_i$, and $V_i/H_i \cong G'$.

Let $i \in \{1,\dots,p^s\}$.
By construction, we have $H_i \leqslant V_i$ and $V_i/H_i \cong G' \cong \Z_{p^{a-1}}^{2s}$.
By Lemma~\ref{lem-quotient}, we have $G'/H_i \cong \Z_{p^{a-1}}^s \times \Z_{p^{a-2}}^s$ and $Q/H_i \cong \Z_p^s$, and $Q/H_i$ is contained in the first $s$ direct factors of~$V_i/H_i$. We also have $(G'/H_i)/(Q/H_i) \cong G'/Q \cong \Z_{p^{a-2}}^{2s}$. Therefore by the inductive hypothesis there exists a $(p^{(a-2)s},p^s)$ LP-partition in $V_i/H_i - G'/H_i$ relative to $Q/H_i$. Then by Theorem~\ref{thm-LPpartitionrecursive} there exists a $(p^s \cdot p^{(a-2)s},p^s)$ LP-partition in $G-V_0$ relative to~$H_0$, so the case $a$ is true.
\end{proof}

We now apply Theorem~\ref{thm-composition} iteratively, making use of the LP-partitions constructed in Theorem~\ref{thm-LPpartitionuniform}, to produce an infinite family of LP-packings in groups of increasing exponent relative to an arbitrary subgroup of the appropriate order. The initial LP-packing in an elementary abelian group is obtained from a spread.

\begin{theorem}\label{thm-LPpackinguniform}
Let $p$ be prime, and let $a,s$ be positive integers. Let $G=\Z_{p^a}^{2s}$ and let $U$ be a subgroup of $G$ of order~$p^{as}$. Then there exists a $(p^{(a-1)s},p^s)$ LP-packing in $G$ relative to~$U$.
\end{theorem}

\begin{proof}
The proof is by induction on $a \ge 1$. The case $a=1$ follows from the construction of a spread in Example~\ref{ex-spread}, because by a suitable choice of generators of $G = \Z_p^{2s}$ we may assume that the element $H_0$ of the spread is~$U$.

Suppose that all cases up to $a-1 \ge 1$ are true.
Since $U$ is a subgroup of $G$ of order $p^{as}$, it has at most $s$ direct factors isomorphic to~$\Z_{p^a}$ and has rank at least~$s$. We may therefore choose $V_0 \cong \Z_{p^{a}}^{s} \times \Z_{p^{a-1}}^{s}$ for which $U \leqslant V_0 \leqslant G$, and also assume that the subgroup of $V_0$ isomorphic to $\Z_{p^a}^s$ is contained in the first $s$ direct factors of~$G$
and has intersection of rank $s$ with~$U$.
Therefore $U$ contains a subgroup $H_0 \cong \Z_p^s$ for which $H_0$ is contained in the first $s$ direct factors of~$G$.
We then have $V_0/H_0 \cong \Z_{p^{a-1}}^{2s}$ and $|U/H_0| = p^{(a-1)s}$.

Then by Theorem~\ref{thm-LPpartitionuniform}, there exists a $(p^{(a-1)s},p^s)$ LP-partition in $G-V_0$ relative to $H_0$. By the inductive hypothesis, there exists a $(p^{(a-2)s},p^s)$ LP-packing in $V_0/H_0$ relative to~$U/H_0$. Therefore by Theorem~\ref{thm-composition}, there exists a $(p^{(a-1)s},p^s)$ LP-packing in $G$ relative to~$U$, so the case $a$ is true.
\end{proof}

The construction process illustrated in Figure~\ref{fig-ladder} shares many features with the recursive construction of difference sets from relative difference sets introduced in~\cite{DJ}, which is illustrated in \cite{DJ-nato} using an analogous representation to Figure~\ref{fig-ladder}. In both settings, an ingredient on the left side is lifted from a factor group and combined with an ingredient on the right side to form a larger example on the left side; and multiple instances of an ingredient on the right side are lifted from factor groups and combined to form a larger example on the right side.
The ingredients in \cite{DJ} corresponding to an LP-packing and an LP-partition are, respectively, a covering extended building set for constructing a difference set, and a building set for constructing a relative difference set.
However, in \cite{DJ} the ingredients are placed into distinct cosets of a subgroup and so can be combined without regard to their intersections as subsets of the subgroup.
In contrast, in the construction of Theorem~\ref{thm-composition}, the lifted subset $P'_j$ of an LP-packing is combined with the subset $R_j$ of an LP-partition by taking their union.
Likewise, in the construction of Theorem~\ref{thm-LPpartitionrecursive}, the subsets of the $i^{\rm th}$ LP-partition are lifted to $S'_{i1},\dots,S'_{it}$, and then
subset $R_j$ of the new LP-partition is formed as the union of $S'_{1j}, \dots, S'_{tj}$ (with one lifted subset derived from each of the $t$ LP-partitions).
This leads to an additional constraint, that the subsets $P'_j$ and $S'_{1j}, \dots, S'_{tj}$ must be disjoint for each~$j$.

We now extend Theorem~\ref{thm-LPpackinguniform} using constructions from Section~\ref{sec3} in order to obtain the result stated in Theorem~\ref{thm-preview}.

\begin{corollary}\label{coro-LPpackingeven}
Let $p$ be prime, let $s_1,\dots,s_v$ be nonnegative integers (not all zero), and let $m=\min\{s_i \mid s_i>0\}$. For each $i=1,\dots,v$, let $G_i=\Z_{p^i}^{2s_i}$ and let $U_i$ be a subgroup of $G_i$ of order~$p^{is_i}$.
Let $G = \prod_{i=1}^v G_i$, let $U = \prod_{i=1}^v U_i$, and let $n = \sqrt{|G|}$.
Then for each $j = 0,\dots,m-1$, there exists an $\big(\frac{n}{p^{m-j}}, p^{m-j}\big)$ LP-packing in $G$ relative to~$U$.
\end{corollary}

\begin{proof}
For each $i=1,\dots,v$ for which $s_i > 0$, by Theorem~\ref{thm-LPpackinguniform} there exists a $(p^{(i-1)s_i},p^{s_i})$ LP-packing in $G_i$ relative to~$U_i$. By a straightforward induction, Corollary~\ref{coro-LPpackingproduct} then gives an
$\big(\frac{n}{p^m}, p^m\big)$ LP-packing in $G$ relative to~$U$.
For $j \in \{0,\dots,m-1\}$, apply Lemma~\ref{lem-LPpackingunion} with $s=p^j$ to give an $\big(\frac{n}{p^{m-j}}, p^{m-j}\big)$ LP-packing in $G$ relative to~$U$.
\end{proof}

\begin{remark}\label{rem-PDSeven}
By Definition~\ref{def-LPpacking}, the constructed LP-packing in Corollary~\ref{coro-LPpackingeven} is a collection of $p^{m-j}$ disjoint regular $\big(n,\frac{n}{p^{m-j}}\big)$ Latin square type PDSs in $G$ avoiding~$U$.
By Lemma~\ref{lem-LPpackingtoPDS}, we also obtain a regular Latin square type PDS in $G$ with parameters $(n,r)$ and $(n,r+1)$, for each positive integer multiple $r$ of $\frac{n}{p^{m-j}}$ satisfying $r \le n$.

This improves on numerous previous constructions of PDSs in three respects: by constructing a collection of disjoint PDSs in $G$ where previously only a single PDS was constructed; by identifying the order~$n$ subgroup~$U$ and showing it can be chosen arbitrarily; and by producing Latin square type PDSs with new parameters~$(n,r)$.
We give an illustrative selection of examples.
\begin{enumerate}[$(i)$]
\item

Result~\ref{res-PCPPDS} constructs a single regular $(n,r)$ Latin square type PDS in $G$ for each positive integer $r \le p^m+1$, whereas the case $j=0$ of Corollary~\ref{coro-LPpackingeven} gives a disjoint collection of $p^m$ such PDSs with $r = \frac{n}{p^m}$.

\item
Result~\ref{res-ringPDS}~$(iii)$ constructs a regular $(n,r)$ Latin square type PDS in $G$ for each positive integer multiple~$r$ of $\frac{n}{p^{\,\gcd(s_1,\dots,s_v)}}$ satisfying $r \le n$, whereas the case $j=0$ of Corollary~\ref{coro-LPpackingeven} gives such a PDS for each positive integer multiple~$r$ of $\frac{n}{p^m}$ satisfying $r \le n$ and whose parameters therefore differ when $m > \gcd(s_1,\dots,s_v)$.

\item
The LP-packing of {\rm \cite[Cor.~6.2]{P08DCC}} described in Table~\ref{tab-LPpacking} is a disjoint collection of $p$ regular $(n,\frac{n}{p})$ PDSs in certain groups $G$ constrained to have at most one direct factor with an exponent that is an odd power of~$p$, whereas the case $j=m-1$ of Corollary~\ref{coro-LPpackingeven} gives such disjoint collections of PDSs without this constraint.

\item
In each group for which the lifting approach in {\rm \cite{H00}, \cite[Sects.~4,~5]{H03}, \cite{HN}} constructs a single PDS, Corollary~\ref{coro-LPpackingeven} provides a collection of disjoint PDSs having the same parameters and avoiding an arbitrary order~$n$ subgroup $U$ of~$G$.

\item
Result~\ref{res-PDSproductLP} constructs a partition of $G-1_G$ into $p$ regular PDSs in $G$, of which $p-1$ are of $(n,\frac{n}{p})$ Latin square type and one is of $\big(n,\frac{n}{p}+1\big)$ Latin square type, where $G$ is constrained to have $s_{2i+1} = 0$ for each~$i > 0$.
In contrast, by applying Remark~\ref{rem-LPpacking}~$(i)$ to the case $j=0$ of Corollary~\ref{coro-LPpackingeven} we obtain a partition of $G-1_G$ into $p^m$ regular PDSs in~$G$, of which $p^m-1$ are of $\big(n,\frac{n}{p^m}\big)$ Latin square type and one is of $\big(n,\frac{n}{p^m}+1\big)$ Latin square type, without constraining~$G$.

\end{enumerate}
\end{remark}

The LP-packings in Corollary~\ref{coro-LPpackingeven} are all contained in $p$-groups of the form $H \times H$, having even rank. We now extend this result to $p$-groups of odd rank, starting from the following infinite family of odd rank examples that was obtained from an elegant construction using Galois rings. (We could likewise make use of the further odd rank example of Example~\ref{ex-LPpacking444}~$(ii)$, relative to a nonelementary abelian group.)

\begin{result}[{\cite[Lem.~3.1]{HLX}}]\label{res-HLX}
Let $p$ be prime and let $w>1$ be an odd integer.
Then there exists a $(p^{w-\s(p,w)},p^{\s(p,w)})$ LP-packing in $\Z_{p^2}^w$ relative to~$(p\Z_{p^2})^w$, where the function $\sigma(p,w)$ is defined in~\eqref{eqn-s}.
\end{result}

\begin{corollary}\label{coro-LPpackingodd}
Let $p$ be prime, let $w>1$ be an odd integer, and let $s_1,\dots, s_v$ be nonnegative integers (possibly all zero).
Let $m = \min\big(\{\s(p,w)\} \cup \{s_i \mid s_i>0\}\big)$.
For each $i=1,\dots,v$, let $G_i=\Z_{p^i}^{2s_i}$ and let $U_i$ be a subgroup of $G_i$ of order~$p^{is_i}$.
Let $G = \Z_{p^2}^{w} \times \prod_{i=1}^v G_i$, let $U = (p\Z_{p^2})^{w} \times \prod_{i=1}^v U_i$, and let $n = \sqrt{|G|}$.
Then for each $j = 0,\dots,m-1$, there exists an $\big(\frac{n}{p^{m-j}}, p^{m-j}\big)$ LP-packing in $G$ relative to~$U$.
\end{corollary}
\begin{proof}
In the case that the $s_i$ are all zero, we have $m = \s(p,w)$ and the result follows by applying Lemma~\ref{lem-LPpackingunion} with $s=p^j$ to Result~\ref{res-HLX}.
Otherwise, at least one of the $s_i$ is positive and we may define \mbox{$m' = \min\{s_i \mid s_i > 0\}$}.
The case $j=0$ of Corollary~\ref{coro-LPpackingeven} then gives an $\big(\frac{n/p^w}{p^{m'}}, p^{m'}\big)$ LP-packing in $\prod_{i=1}^v G_i$ relative to~$\prod_{i=1}^v U_i$.
Combine with the LP-packing of Result~\ref{res-HLX} using Corollary~\ref{coro-LPpackingproduct} to give an $\big(\frac{n}{p^{m}}, p^m\big)$ LP-packing in $G$ relative to~$U$.
Then for $j \in \{0,\dots,m-1\}$, apply Lemma~\ref{lem-LPpackingunion} with $s=p^j$ to give an $\big(\frac{n}{p^{m-j}}, p^{m-j}\big)$ LP-packing in $G$ relative to~$U$.
\end{proof}

\begin{remark}\label{rem-PDSodd}
By Definition~\ref{def-LPpacking}, the constructed LP-packing in Corollary~\ref{coro-LPpackingodd} is a collection of $p^{m-j}$ disjoint regular $\big(n,\frac{n}{p^{m-j}}\big)$ Latin square type PDSs in $G$ avoiding~$U$.
By Lemma~\ref{lem-LPpackingtoPDS}, we also obtain a regular Latin square type PDS in $G$ with parameters $\big(n,r)$ and with parameters $\big(n,r+1)$, for each positive integer multiple $r$ of $\frac{n}{p^{m-j}}$ satisfying $r \le n$.

We can show that this improves on previous constructions along similar lines to Remark~\ref{rem-PDSeven}. For example,
Result~\ref{res-ringPDS}~$(ii)$ constructs a single regular $(n,r)$ Latin square type PDS in $G$ for $v=1$ and $i =1$ and $s_1$ a positive integer multiple of $w$, for each positive integer multiple $r$ of $\frac{n}{p^{\s(p,w)}}$ satisfying $r \le n$.
In contrast, the case $j=0$ of Corollary~\ref{coro-LPpackingodd} gives a disjoint collection of $p^m$ such PDSs in much more general groups $G$ (having arbitrary $v$ and~$s_i$).
\end{remark}

\section{NLP-packings}\label{sec6}

In this section, we introduce a $(c,t-1)$ NLP-packing as a collection of $t-1$ disjoint regular $(tc,c)$ negative Latin square type PDSs, in an abelian group of order $t^2c^2$, whose union satisfies an additional property.
This structure forms a counterpart to a $(c,t)$ LP-packing.
We provide examples of NLP-packings, and show how to combine them with LP-packings via a product construction to form NLP-packings in the direct product of the starting groups.
In conjunction with the LP-packings of Corollaries~\ref{coro-LPpackingeven} and~\ref{coro-LPpackingodd}, this produces infinite families of NLP-packings.
However, we do not have a lifting result for NLP-packings similar to Theorem~\ref{thm-composition}.

\begin{definition}[NLP-packing]\label{def-NLPpacking}
Let $t > 1$ and $c>0$ be integers, and let $G$ be an abelian group of order $t^2c^2$.
A \emph{$(c,t-1)$ NLP-packing in $G$} is a collection $\{P_1,\dots,P_{t-1}\}$ of $t-1$ regular $(tc,c)$ negative Latin square type PDSs in~$G$ for which $G-1_G-\sum_{i=1}^{t-1}P_i$ is a regular $(tc,c-1)$ negative Latin square type PDS in~$G$.
\end{definition}

\begin{remark}\label{rem-NLPpacking}
\quad
\begin{enumerate}[$(i)$]
\item
We can rephrase Definition~\ref{def-NLPpacking} to say that a $(c,t-1)$ NLP-packing in $G$ is a partition of $G-1_G$ into $t$ regular partial difference sets, of which $t-1$ are of $(tc,c)$ negative Latin square type and one is of $(tc,c-1)$ negative Latin square type. Previous results on NLP-packings use this phrasing.

\item
In Definition~\ref{def-NLPpacking}, each $P_i$ is a $c(tc+1)$-subset of $G$, and $G-1_G-\sum_{i=1}^{t-1}P_i$ is a $(c-1)(tc+1)$-subset of $G$, so the subsets $P_i$ of a $(c,t-1)$ NLP-packing are necessarily disjoint.
For all $c \ge 1$, if a $(c,t-1)$ NLP-packing exists in $G$, then it contains the maximum possible number of disjoint regular $(tc,c)$ negative Latin square type PDSs.

\item
The case $r=0$ was allowed in Definition~\ref{def-LatinPDS} so that the definition of a $(c,t-1)$ NLP-packing applies when $c=1$.
\end{enumerate}
\end{remark}

NLP-packings have been studied in a series of previous papers, as summarized in Table~\ref{tab-NLPpacking}. We now characterize a $(c,t-1)$ NLP-packing using character sums.

\begin{lemma}\label{lem-NLPpacking}
Let $P_1,\dots,P_{t-1}$ be disjoint $c(tc+1)$-subsets of an abelian group $G$ of order~$t^2c^2$.
Then $\{P_1, \ldots, P_{t-1}\}$ is a $(c,t-1)$ NLP-packing in $G$ if and only if the multiset equality
\[
\{\chi(P_1),\dots,\chi(P_{t-1})\} = \{c,\dots,c\} \mbox{ or } \{-(t-1)c,c,\dots,c\}
\]
holds for all nonprincipal characters $\chi$ of~$G$.
\end{lemma}
\begin{proof}
Each $P_i$ is a $c(tc+1)$-subset of $G$, and the $P_i$ are disjoint, so $G-1_G-\sum_{i=1}^{t-1}P_i$ is a $(c-1)(tc+1)$-subset of~$G$. Lemma~\ref{lem-LatinPDS} then shows that $\{P_1,\dots,P_{t-1}\}$ is a $(c,t-1)$ NLP-packing in $G$ if and only if
\begin{equation}
\label{eqn-chiPi}
\chi(P_i) \in \{c, -(t-1)c\} \quad \! \mbox{and} \! \quad
\chi\Big(G-1_G-\sum_{i=1}^{t-1}P_i\Big) \in \{c-1, -(tc-c+1)\}
\end{equation}
for all nonprincipal characters $\chi$ of $G$.
For each nonprincipal character $\chi$ of $G$, condition \eqref{eqn-chiPi}
is equivalent to
\[
\chi(P_i) \in \{c, -(t-1)c\} \quad \mbox{ and } \quad
\sum_{i=1}^{t-1}\chi(P_i) \in \{-c, (t-1)c)\},
\]
which in turn is equivalent to
\[
\{\chi(P_1),\dots,\chi(P_{t-1})\} = \{c,\dots,c\} \mbox{ or } \{-(t-1)c,c,\dots,c\},
\]
as required.
\end{proof}

Lemmas~\ref{lem-LatinPDS} and~\ref{lem-NLPpacking} allow us to combine subsets from an NLP-packing to obtain negative Latin square type PDSs with various parameters.

\begin{lemma}\label{lem-NLPpackingtoPDS}
Suppose that $\{P_1, \ldots, P_{t-1}\}$ is a $(c,t-1)$ NLP-packing in an abelian group $G$ of order~$t^2c^2$, and let $I$ be a $b$-subset of $\{1,\dots,t-1\}$. Then $\sum_{i \in I} P_i$ is a regular $(tc,bc)$ negative Latin square type PDS in~$G$.
\end{lemma}

\begin{remark}
Suppose that $\{P_1, \ldots, P_{t-1}\}$ is a $(c,t-1)$ NLP-packing in an abelian group $G$ of order~$t^2c^2$.

\begin{enumerate}[$(i)$]

\item
Let $I$ be a $b$-subset of $\{1,\dots,t-1\}$.
By Lemma~\ref{lem-LatinPDS}, an $r(n+1)$-subset $D$ of an abelian group $G$ of order $n^2$ is a regular $(n,r)$ negative Latin square type PDS in $G$ if and only if 
$G-1_G-D$ is a regular $(n,n-r-1)$ negative Latin square type PDS in~$G$.
Lemma~\ref{lem-NLPpackingtoPDS} then shows that
$G-1_G-\sum_{i \in I} P_i$ is a regular $(tc,(t-b)c-1)$ negative Latin square type PDS in~$G$.

\item
For each $i$, let $\wti{{P\mkern 0mu}_i} = \{\chi \in \wh{G} \mid \chi(P_i) = -(t-1)c\}$.
Let $P = G-1_G-\sum_{i=1}^{t-1}P_i$ and let $\wti{P} = \{\chi \in \wh{G} \mid \chi(P) = -(tc-c+1)\}$.
By Lemma~\ref{lem-LatinPDS} and Definition~\ref{def-NLPpacking}, each $\wti{{P\mkern 0mu}_i}$ is a regular $(tc,c)$ negative Latin square type PDS in~$\wh{G}$, and $\wti{P}$ is a
regular $(tc,c-1)$ negative Latin square type PDS in~$\wh{G}$.
By Lemma~\ref{lem-NLPpacking}, each nonprincipal character $\chi$ of $G$ belongs to exactly one of $\wti{P_1},\dots,\wti{P_{t-1}},\wti{P}$.
By Definition~\ref{def-NLPpacking}, the collection $\{\wti{P_1}, \dots, \wti{P_{t-1}}\}$ is therefore a $(c,t-1)$ NLP-packing in~$\wh{G}$.
\end{enumerate}

\end{remark}

We give some small examples of NLP-packings.

\begin{example}\label{ex-NLPpackingsmall}
\quad
\begin{enumerate}[$(i)$]
\item
{\rm\cite[p.~373]{P08DCC}}.
A $(1,3)$ NLP-packing $\{P_1,P_2,P_3\}$ in $\Z_2^4=\lan x,y,z,w \ran$ is given by
\begin{align*}
P_1 &= x+z+yw+xw+yz, \\
P_2 &= y+w+xyzw+yzw+xyw, \\
P_3 &= xy+zw+xz+xyz+xzw.
\end{align*}

\item
{\rm\cite[p.~373]{P08DCC}}.
A $(1,3)$ NLP-packing $\{P_1,P_2,P_3\}$ in $\Z_4^2=\lan x,y \ran$ is given by
\begin{align*}
P_1 &= x^2+y+y^3+xy+x^3y^3, \\
P_2 &= x+x^3+y^2+xy^3+x^3y, \\
P_3 &= xy^2+x^3y^2+x^2y+x^2y^3+x^2y^2.
\end{align*}

\item
{\rm\cite[Ex.~1]{P08DCC}}.
A $(1,2)$ NLP-packing $\{P_1,P_2\}$ in $\Z_3^2=\lan x,y \ran$ is given by
\begin{align*}
P_1 &= x+x^2+y+y^2, \\
P_2 &= xy+x^2y^2+xy^2+x^2y.
\end{align*}

\end{enumerate}

\end{example}

We next show that combining the subsets of an NLP-packing into equally-sized collections gives an NLP-packing with fewer subsets.

\begin{lemma}\label{lem-NLPpackingunion}
Suppose there exists a $(c,t-1)$ NLP-packing in an abelian group $G$ of order $t^2c^2$, and suppose $s$ divides~$t$. Then there exists an $(sc,\frac{t}{s}-1)$ NLP-packing in~$G$.
\end{lemma}
\begin{proof}
Let $\{P_1, \dots, P_{t-1}\}$ be a $(c,t-1)$ NLP-packing in $G$, and let
\[
P'_i = P_{is+1} + P_{is+2} + \dots + P_{is+s} \quad \mbox{for $i = 0, 1,\dots, \frac{t}{s}-2$}.
\]
Then Lemma~\ref{lem-NLPpacking} shows that $\{P'_0,\dots,P'_{\frac{t}{s}-2}\}$ is an $(sc,\frac{t}{s}-1)$ NLP-packing in~$G$.
\end{proof}

The following product construction combines an NLP-packing with an LP-packing to form an NLP-packing in a larger group.

\begin{theorem}[{\cite[Thm.~2.1]{P08DCC}}]\label{thm-NLPLPpackingproduct}
Suppose there exists a
$(c_1,t-1)$ NLP-packing in an abelian group $G_1$ of order $t^2c_1^2$, and a
$(c_2,t)$ LP-packing in an abelian group $G_2$ of order $t^2c_2^2$ relative to a subgroup $U_2$ of order~$tc_2$.
Then there exists a $(tc_1c_2,t-1)$ NLP-packing in $G_1 \times G_2$.
\end{theorem}
\begin{proof}
Let $\{ P_{1,0},\dots,P_{1,t-2} \}$ be a $(c_1,t-1)$ NLP-packing in~$G_1$, and let  $\{ P_{2,0},\dots,P_{2,t-1} \}$ be a $(c_2,t)$ LP-packing in~$G_2$ relative to~$U_2$.
Define $P_{1,t-1}=G_1-\sum_{i=0}^{t-2} P_{1,i}$ and
\[
K_{\ell}=P_{1,\ell}U_2+\sum_{i=0}^{t-1} P_{1,i}P_{2,i+\ell} \quad \mbox{for $\ell = 0, \dots, t-2$},
\]
where the subscript $i+\ell$ is reduced modulo~$t$.
We shall use Lemma~\ref{lem-NLPpacking} to show that
$\{ K_0, K_1, \dots,$ $K_{t-2} \}$ is a $(tc_1c_2,t-1)$ NLP-packing in $G_1 \times G_2$.

The $P_{1,i}$ are disjoint $c_1(tc_1+1)$-subsets of $G_1$ for $i=0,\dots,t-2$, and by definition $P_{1,t-1}$ is a $c_1(tc_1-t+1)$-subset of $G_1$ disjoint from the other~$P_{1,i}$. The $P_{2,i}$ are $c_2(tc_2-1)$-subsets of $G_2-U_2$. Therefore the $K_\ell$ are disjoint subsets of $G_1 \times G_2$, each of size
$c_1(tc_1+1)tc_2+c_1(tc_1-t+1)c_2(tc_2-1)+(t-1)c_1(tc_1+1)c_2(tc_2-1)=tc_1c_2(t^2c_1c_2+1)$.
Let $\chi$ be a nonprincipal character of $G_1 \times G_2$, and let $\chi_j=\chi|_{G_j}$ for $j=1,2$. Then
\[
\chi(K_\ell) = \chi_1(P_{1,\ell}) \chi_2(U_2) + \sum_{i=0}^{t-1} \chi_1(P_{1,i})\chi_2(P_{2,i+\ell}) \quad \mbox{for $\ell = 0,\dots,t-2$}.
\]
By Lemma~\ref{lem-NLPpacking}, it remains to prove that
\begin{equation}\label{eqn-NLPprod}
\{\chi(K_0),\dots,\chi(K_{t-2})\} = \{tc_1c_2,\dots,tc_1c_2\} \mbox{ or } \{-t(t-1)c_1c_2,tc_1c_2,\dots,tc_1c_2\}.
\end{equation}

We are given that
\begin{equation}\label{eqn-chi1P10}
\{\chi_1(P_{1,0}),\dots,\chi_1(P_{1,t-2})\} = \begin{cases}
  \{c_1(tc_1+1),\dots,c_1(tc_1+1)\}	& \mbox{if $\chi_1 \in G_1^\perp$,} \\
  \{c_1,\dots,c_1\} \mbox{ or } \{-(t-1)c_1,c_1,\dots,c_1\}		
					& \mbox{if $\chi_1 \not\in G_1^\perp$}
\end{cases}
\end{equation}
by Lemma~\ref{lem-NLPpacking}, and that
\[
\{\chi_2(P_{2,0}),\dots,\chi_2(P_{2,t-1})\} = \begin{cases}
  \{c_2(tc_2-1),\dots, c_2(tc_2-1)\} 	& \mbox{if $\chi_2 \in G_2^\perp$,} \\
  \{-c_2,\dots, -c_2\} 	    		& \mbox{if $\chi_2 \in U_2^\perp \sm G_2^\perp$,} \\
  \{(t-1)c_2, -c_2, \dots, -c_2\} 	& \mbox{if $\chi_2 \notin U_2^\perp$}
\end{cases}
\]
by Lemma~\ref{lem-LPpacking}.
The following two cases together establish~\eqref{eqn-NLPprod}.

\begin{description}
\item[{\bf Case 1:}] $\chi_2$ is principal on $U_2$.\\
Then $\chi_2(P_{2,i})$ is constant over $i$ (regardless of whether $\chi_2$ is principal or nonprincipal on $G_2$), so
$\chi(K_\ell) = \chi_1(P_{1,\ell}) tc_2 + \chi_1 \big(\sum_{i=0}^{t-1}P_{1,i}\big) \chi_2(P_{2,1}) = \chi_1(P_{1,\ell}) tc_2 + \chi_1(G_1) \chi_2(P_{2,1})$.
In the case that $\chi_1$ is principal on $G_1$, and therefore $\chi_2$ is nonprincipal on $G_2$, we obtain
$\chi(K_\ell) = c_1(tc_1+1) tc_2 + t^2c_1^2 (-c_2) = tc_1c_2$ for each~$\ell$.
Otherwise $\chi_1$ is nonprincipal on $G_1$, and then we obtain
$\chi(K_\ell) = \chi_1(P_{1,\ell}) tc_2$ so that
$\{\chi(K_0),\dots,\chi(K_{t-2})\} =
\{tc_1c_2,\dots,tc_1c_2\}$ or $\{-(t-1)tc_1c_2,tc_1c_2,\dots,tc_1c_2\}$.

\item[{\bf Case 2:}] $\chi_2$ is nonprincipal on $U_2$.\\
Then
$\chi(K_\ell) = \sum_{i=0}^{t-1} \chi_1(P_{1,i})\chi_2(P_{2,i+\ell})$.
By definition of $P_{1,t-1}$, we can extend \eqref{eqn-chi1P10} to obtain
\[
\begin{split}
&\{\chi_1(P_{1,0}),\dots,\chi_1(P_{1,t-1})\} \\
=& \begin{cases}
  \{c_1(tc_1-t+1),c_1(tc_1+1),\dots,
  c_1(tc_1+1)\}	
  & \mbox{if $\chi_1 \in G_1^\perp$,} \\
  \{-(t-1)c_1,c_1,\dots,c_1\} & \mbox{if $\chi_1 \not\in G_1^\perp$},
\end{cases}
\end{split}
\]
which can be written as
\[
\{\chi_1(P_{1,0}),\dots,\chi_1(P_{1,t-1})\} = \{d-tc_1,d,\dots,d\} \quad \mbox{for all $\chi_1$}
\]
(where the integer $d$ equals $ c_1(tc_1+1)$ if $\chi_1$ is principal on $G_1^\perp$, and equals $c_1$ otherwise).
Therefore the multiset $\{\chi(K_0),\dots,\chi(K_{t-2})\}$ contains
either one or zero occurrences of the value
$(d-tc_1)(t-1)c_2 + (t-1)d(-c_2) = -t(t-1)c_1c_2$,
and (respectively)
$t-2$ or $t-1$ occurrences of the value $(d-tc_1)(-c_2)+d(t-1)c_2 + (t-2)d(-c_2) = tc_1c_2$.
\end{description}
\end{proof}
\noindent

\begin{remark}
The construction of Theorem~\ref{thm-NLPLPpackingproduct} is a restatement (without restricting $t$ to be a prime power) of the product construction for collections of negative Latin square type PDSs and Latin square type PDSs given in {\rm \cite[Thm.~2.1]{P08DCC}}.
Indeed, the configuration described in {\rm \cite[Lemma 2.1]{P08DCC}} for $e=-1$ can be represented as $\{ P_1, \dots, P_{t-1}, G-1_G-\sum_{i=1}^{t-1} P_i\}$, where $\{P_1,\dots,P_{t-1}\}$ is a $(c,t-1)$ NLP-packing in $G$.
Although our formulation explicitly identifies the subgroup $U_2$ of the starting LP-packing, this does not seem to confer an advantage because the subgroup does not appear in the conclusion of the theorem.
\end{remark}

We now extend the construction of Theorem~\ref{thm-NLPLPpackingproduct} using Lemmas~\ref{lem-LPpackingunion} and~\ref{lem-NLPpackingunion}.

\begin{corollary}\label{coro-NLPLPpackingproduct}
Suppose there exists a $(c_1,t_1-1)$ NLP-packing in an abelian group $G_1$ of order $t_1^2c_1^2$, and there exists a $(c_2,t_2)$ LP-packing in an abelian group $G_2$ of order $t_2^2c_2^2$ relative to a subgroup $U_2$ of order~$t_2c_2$.
\begin{enumerate}[$(i)$]
\item
Suppose that $t_1$ divides~$t_2$. Then there exists a $(t_2c_1c_2,t_1-1)$ NLP-packing in $G_1 \times G_2$.
\item
Suppose that $t_2$ divides~$t_1$. Then there exists a $(t_1c_1c_2,t_2-1)$ NLP-packing in $G_1 \times G_2$.
\end{enumerate}
\end{corollary}
\begin{proof}
\quad
\begin{enumerate}[$(i)$]
\item
By Lemma~\ref{lem-LPpackingunion} with $s=\frac{t_2}{t_1}$, there exists a $(\frac{t_2}{t_1}c_2,t_1)$ LP-packing in $G_2$ relative to $U_2$. Then by Theorem~\ref{thm-NLPLPpackingproduct}, there exists a $(t_2c_1c_2,t_1-1)$ NLP-packing in $G_1 \times G_2$.
\item
By Lemma~\ref{lem-NLPpackingunion} with $s=\frac{t_1}{t_2}$, there exists a $(\frac{t_1}{t_2}c_1,t_2-1)$ NLP-packing in $G_1$. Then by Theorem~\ref{thm-NLPLPpackingproduct}, there exists a $(t_1c_1c_2,t_2-1)$ NLP-packing in $G_1 \times G_2$.
\end{enumerate}
\end{proof}

The product constructions of Theorems~\ref{thm-LPpackingproduct} and~\ref{thm-NLPLPpackingproduct} both combine two packings to form a packing in the direct product of the starting groups. Theorem~\ref{thm-LPpackingproduct} combines two LP-packings to form an LP-packing, whereas Theorem~\ref{thm-NLPLPpackingproduct} combines an NLP-packing and an LP-packing to form an NLP-packing.
We note in passing that we can likewise combine two NLP-packings to form a collection of regular Latin square type PDSs that partitions the nonidentity elements of the product group.

\begin{proposition}[{\cite[Thm.~2.3]{P08DCC}}]\label{prop-NLPpackingproduct}
For $j=1,2$, suppose there exists a $(c_j,t-1)$ NLP-packing in an abelian group $G_j$ of order~$t^2c_j^2$. Then there exists a collection $\{K_1,\dots,K_{t-1}\}$ of $t-1$ disjoint regular $(t^2c_1c_2,tc_1c_2)$ Latin square type PDSs in $G_1 \times G_2$ for which $G_1 \times G_2 -1_{G_1 \times G_2}-\sum_{\ell=1}^{t-1}K_\ell$ is a regular $(t^2c_1c_2,tc_1c_2+1)$ Latin square type PDS in~$G_1 \times G_2$.
\end{proposition}
\begin{proof}[Proof (Outline)]
For $j=1,2$, let $\{P_{j,0},\dots,P_{j,t-2}\}$ be a $(c_j,t-1)$ NLP-packing in~$G_j$ and define $P_{j,t-1} = G_j -\sum_{i=0}^{t-2}P_{j,i}$. Define
\[
K_\ell = \sum_{i=0}^{t-1} P_{1,i}P_{2,i+\ell} \quad \mbox{for $\ell = 1,\dots,t-1$},
\]
where the subscript $i+\ell$ is reduced modulo~$t$.
Model the rest of the proof on that of Theorem~\ref{thm-NLPLPpackingproduct},
distinguishing the case that $\chi|_{G_1}$ is principal on $G_1$ from the case that $\chi|_{G_1}$ is nonprincipal on~$G_1$, and applying Lemma~\ref{lem-LatinPDS}.
\end{proof}

\begin{remark}\label{rem-productconstructions}
In the construction of Proposition~\ref{prop-NLPpackingproduct}, if we could identify an order $t^2c_1c_2$ subgroup $U$ of $G_1 \times G_2$ contained in $G_1 \times G_2 - \sum_{\ell=1}^{t-1}K_\ell$ for which
$G_1 \times G_2 -U - \sum_{\ell=1}^{t-1}K_\ell$ is a regular $(t^2c_1c_2,tc_1c_2)$ Latin square type PDS in $G_1 \times G_2$, then we could conclude that $\{ K_1, \ldots, K_{t-1}, G_1 \times G_2 - U - \sum_{\ell=1}^{t-1}K_\ell \}$ is a $(tc_1c_2,t)$ LP-packing in $G_1 \times G_2$ relative to $U$. However, we do not know how to identify such a subgroup~$U$ (see also Remark~\ref{rem-identifyU}).
\end{remark}

We now apply Theorem~\ref{thm-NLPLPpackingproduct} to the small NLP-packings of Example~\ref{ex-NLPpackingsmall} and LP-packings from Corollaries~\ref{coro-LPpackingeven} and~\ref{coro-LPpackingodd} to produce two infinite families of NLP-packings.

\begin{corollary}\label{coro-NLPpacking}
Let $u,w$ be nonegative integers. Let $s_3,\dots,s_v$ be nonnegative integers (possibly all zero).

\begin{enumerate}[$(i)$]

\item
Let $u \ne 1$, let $w \not\in \{1,2,3\}$ and let each $s_i \ne 1$.
Let $G$ be either $\Z_2^{2u+4} \times \Z_4^w \times \prod_{i=3}^v \Z_{2^i}^{2s_i}$ or $\Z_2^{2u} \times \Z_4^{w+2} \times \prod_{i=3}^v \Z_{2^i}^{2s_i}$, and let $n=\sqrt{|G|}$.
Then there exists an $\big(\frac{n}{4},3\big)$ NLP-packing in~$G$.

\item
Let $u \ne 0$ and let $w \ne 1$. Let $G = \Z_3^{2u} \times \Z_9^w \times \prod_{i=3}^v \Z_{3^i}^{2s_i}$ and let $n = \sqrt{|G|}$.
Then there exists an $\big(\frac{n}{3},2\big)$ NLP-packing in~$G$.

\end{enumerate}
\end{corollary}

\begin{proof}
\quad
\begin{enumerate}[$(i)$]

\item
In the case that $u,w$, and each $s_i$ is zero, the result is given by Example~\ref{ex-NLPpackingsmall}~$(i),(ii)$. Otherwise, let $G' = \Z_2^{2u} \times \Z_4^w \times \prod_{i=3}^v \Z_{2^i}^{2s_i}$ and construct an $\big(\frac{n/4}{4},4\big)$ LP-packing in $G'$ relative to a subgroup of order~$n/4$.
To do so when $w$ is even, take $p=2$ and $(s_1,s_2) = (u,\frac{w}{2})$ and $j = m-2 \ge 0$ in Corollary~\ref{coro-LPpackingeven}.
To do so when $w$ is odd, note that $\s(2,w) = \frac{w-1}{2} \ge 2$ and take $p=2$ and $(s_1,s_2) = (u,0)$ and $j=m-2 \ge 0$ in Corollary~\ref{coro-LPpackingodd}.

Now let $H$ be either $\Z_2^4$ or $\Z_4^2$.  Use Theorem~\ref{thm-NLPLPpackingproduct} to combine the $(1,3)$ NLP-packing in~$H$ given by Example~\ref{ex-NLPpackingsmall}~$(i),(ii)$ with the constructed $\big(\frac{n/4}{4},4\big)$ LP-packing in $G'$ to produce the required $\big(\frac{n}{4},3\big)$ NLP-packing in $H \times G'$.

\item
In the case that $u=1$, and $w$ and each of the $s_i$ is zero, the result is given by Example~\ref{ex-NLPpackingsmall}~$(iii)$.
Otherwise, let $G' = \Z_3^{2u-2} \times \Z_9^w \times \prod_{i=3}^v \Z_{3^i}^{2s_i}$ and construct an $\big(\frac{n/3}{3},3\big)$ LP-packing in $G'$ relative to a subgroup of order $n/3$.
To do so when $w$ is even, take $p=3$ and
$(s_1,s_2) = (u-1,\frac{w}{2})$ and $j = m-1 \ge 0$ in Corollary~\ref{coro-LPpackingeven}.
To do so when $w$ is odd, note that $\s(3,w) \ge 1$ and take $p = 3$ and $(s_1,s_2) = (u-1,0)$ and $j=m-1 \ge 0$ in Corollary~\ref{coro-LPpackingodd}.

Now use Theorem~\ref{thm-NLPLPpackingproduct} to combine the $(1,2)$ NLP-packing in~$\Z_3^2$ given by Example~\ref{ex-NLPpackingsmall}~$(iii)$ with the constructed $\big(\frac{n/3}{3},3\big)$ LP-packing in $G'$ to produce the required $\big(\frac{n}{3},2\big)$ NLP-packing in~$G$.

\end{enumerate}
\end{proof}

\begin{remark}\label{rem-PDSproduct}
Using Remark~\ref{rem-NLPpacking}~$(i)$, we can rephrase Result~\ref{res-PDSproductNLP} in terms of $(\frac{n}{t},t-1)$ NLP-packings in groups of order~$n^2$. As we now describe, we then see that Corollary~\ref{coro-NLPpacking} is significantly more general than Result~\ref{res-PDSproductNLP}.

Result~\ref{res-PDSproductNLP}~$(i)$ and Corollary~\ref{coro-NLPpacking}~$(i)$ each construct NLP-packings with parameters $\big(\frac{n}{4},3\big)$. The former result is restricted to groups of the form $\Z_2^{4u} \times \Z_4^{2w} \times \prod_{i=2}^v \Z_{2^{2i}}^{4s_{2i}}$ for $u+w\ge 1$, whereas the latter result applies to the forms
$\Z_2^{2u+4} \times \Z_4^w \times \prod_{i=3}^v \Z_{2^i}^{2s_i}$ and
$\Z_2^{2u} \times \Z_4^{w+2} \times \prod_{i=3}^v \Z_{2^i}^{2s_i}$
for $u \ne 1$ and $w \not\in \{1,2,3\}$ and each $s_i \ne 1$.

Similarly, Result~\ref{res-PDSproductNLP}~$(ii)$ and Corollary~\ref{coro-NLPpacking}~$(ii)$ each construct NLP-packings with parameters $\big(\frac{n}{3},2\big)$. The former result is restricted to groups of the form
$\Z_3^{2u+2} \times \prod_{i=1}^v \Z_{3^{2i}}^{2s_{2i}}$,
whereas the latter result applies to the form
$\Z_3^{2u} \times \Z_9^w \times \prod_{i=3}^v \Z_{3^i}^{2s_i}$ for $u \ne 0$ and $w \ne 1$.
\end{remark}

\section{Relationship with Strongly Regular Bent Functions}\label{sec7}
In this section, we examine how LP-packings and NLP-packings are related to the hyperbolic and elliptic strongly regular bent functions introduced by Chen and Polhill in~\cite{CP13}.

\begin{definition}[{\cite[Thm.~3.7]{CP13}}]\label{def-srbf}
Let $G$ and $H$ be abelian groups. Let $f$ be a surjective function from $G$ to $H$ satisfying $f(1_G)=1_H$, and let $f^{-1}$ be the pre-image of~$f$.
\begin{enumerate}[$(i)$]
\item
$f$ is a \emph{hyperbolic strongly regular bent function} if
$f^{-1}(h)$ is a regular $\big(\sqrt{|G|},\frac{\sqrt{|G|}}{|H|}\big)$ Latin square type PDS for each $h \ne 1_H$, and
$f^{-1}(1_H)-1_G$ is a regular $\big(\sqrt{|G|},\frac{\sqrt{|G|}}{|H|}+1\big)$ Latin square type PDS.

\item
$f$ is an \emph{elliptic strongly regular bent function} if
$f^{-1}(h)$ is a regular $\big(\sqrt{|G|},\frac{\sqrt{|G|}}{|H|}\big)$ negative Latin square type PDS for each $h \ne 1_H$, and
$f^{-1}(1_H) - 1_G$ is a regular $\big(\sqrt{|G|},\frac{\sqrt{|G|}}{|H|}-1\big)$ negative Latin square type PDS.
\end{enumerate}
\end{definition}

We can relate hyperbolic strongly regular bent functions to LP-packings, and elliptic strongly regular bent functions to NLP-packings.
\begin{proposition}\label{prop-bent}
Let $G$ be an abelian group of order $t^2c^2$, let $U$ be a subgroup of $G$ of order~$tc$, and let $H$ be an abelian group of order~$t$.
\begin{enumerate}[$(i)$]
\item
If there exists a $(c,t)$ LP-packing in $G$ relative to~$U$, then there exists a hyperbolic strongly regular bent function from $G$ to~$H$.

\item
The existence of a $(c,t-1)$ NLP-packing in $G$ is equivalent to the existence of an elliptic strongly regular bent function from $G$ to~$H$.
\end{enumerate}

\end{proposition}

\begin{proof}
Write $H=\{1_H,h_1,h_2,\ldots,h_{t-1}\}$.
\begin{enumerate}[$(i)$]
\item
Let $\{P_1,\dots,P_t\}$ be a $(c,t)$ LP-packing in $G$ relative to~$U$, and define the surjective function $f:G \rightarrow H$ by
\[
    f(g)=\begin{cases}
      h_i & \mbox{if $g \in P_i$ for $i \in \{1,\dots,t-1\}$,} \\
      1_H & \mbox{if $g \in P_t + U$}.
    \end{cases}
\]
Then by Remark~\ref{rem-LPpacking}~$(i)$, $f$ is a hyperbolic strongly regular bent function.

\item
Suppose firstly that $\{P_1,\dots,P_{t-1}\}$ is a $(c,t-1)$ NLP-packing in $G$, and define the surjective function $f:G \rightarrow H$ by
\[
    f(g)=\begin{cases}
      h_i & \mbox{if $g \in P_i$,} \\
      1_H & \mbox{otherwise}.
    \end{cases}
\]
Then by Definition~\ref{def-NLPpacking}, $f$ is an elliptic strongly regular bent function from $G$ to~$H$.

Conversely, suppose that $f:G \rightarrow H$ is an elliptic strongly regular bent function, and define subsets $P_1,\dots,P_{t-1}$ of $G$ by
\[
P_i = f^{-1}(h_i) \quad \mbox{for $i=1,\dots,t-1$}.
\]
Then by Definition~\ref{def-NLPpacking}, $\{P_1,\dots,P_{t-1}\}$ is a $(c,t-1)$ NLP-packing in~$G$ from $G$ to~$H$.
\end{enumerate}
\end{proof}

\begin{remark}\label{rem-identifyU}
Proposition~\ref{prop-bent} shows the equivalence of an NLP-packing and an elliptic strongly regular bent function, but not the equivalence of an LP-packing and a hyperbolic strongly regular bent function~$f$.
In order to establish the converse to Proposition~\ref{prop-bent}~$(i)$, we would need to identify an order $\sqrt{|G|}$ subgroup $U$ of $G$ contained in $f^{-1}(1_H)$ for which $f^{-1}(1_H)-U$ is a regular $\big(\sqrt{|G|},\frac{\sqrt{|G|}}{|H|}\big)$ Latin square type PDS in~$G$.
The additional conditions on $f$ required to guarantee the existence of such a subgroup~$U$ could be substantial, as we discuss in Example~\ref{ex-ternarybent}~$(i)$ below.
\end{remark}

The following example uses strongly regular bent functions to construct collections of PDSs, and provides an alternative derivation of a special case of Corollary~\ref{coro-NLPpacking}~$(ii)$.
\begin{example}[{\rm \cite[Thm.~1, Cor.~3]{TPF}}]\label{ex-ternarybent}
Let $p$ be a prime, and let $\zp$ be a primitive $p^{\rm {th}}$ root of unity.
Let $g$ be a function from $\F_{p^n}$ to~$\F_p$, and let
$\Tr$ be the trace function from $\F_{p^n}$ to~$\F_p$.
The \emph{Walsh transform} $\cW_g : \F_{p^n} \rightarrow \CC$ of $g$ is given by
$$
\cW_g(b)=\sum_{x \in \F_{p^n}} \zp^{g(x)+\Tr(bx)}.
$$
The function $g$ is \emph{bent} if
$|\cW_g(b)|=p^{\frac{n}{2}}$ for all $b \in \F_{p^n}$.
A bent function $g$ is \emph{weakly regular} if there exists a function $g^*: \F_{p^n} \rightarrow \F_p$ satisfying
\[
\cW_g(b)=\mu p^{\frac{n}{2}}\zp^{g^*(b)} \quad \mbox{for some $\mu \in \CC$ with $|\mu|=1$ and for all $b \in \F_{p^n}$},
\]
and is \emph{regular} (a much stronger condition) if there exists a function $g^*: \F_{p^n} \rightarrow \F_p$ such that $\cW_g(b)= p^{\frac{n}{2}}\zp^{g^*(b)}$ for all $b \in \F_{p^n}$.

Suppose $f: \F_{3^{2s}} \rightarrow \F_3$ is a weakly regular
bent function satisfying $f(-x)=f(x)$ and $f(0)=0$.
Then exactly one of the following holds:
\begin{enumerate}[$(i)$]
\item
$|f^{-1}(0)|=3^{2s-1}+2\cdot3^{s-1}$, and $f$ is a hyperbolic strongly regular bent function from $\F_{3^{2s}}$ to~$\F_3$:
each of $f^{-1}(1)$ and $f^{-1}(2)$ is a regular $(3^s,3^{s-1})$ Latin square type PDS in $\F_{3^{2s}}$, and
$f^{-1}(0) \sm \{0\}$ is a regular $(3^s,3^{s-1}+1)$ Latin square type PDS in $\F_{3^{2s}}$.

\item
$|f^{-1}(0)|=3^{2s-1}-2\cdot3^{s-1}$, and $f$ is an elliptic strongly regular bent function from $\F_{3^{2s}}$ to $\F_3$:
each of $f^{-1}(1)$ and $f^{-1}(2)$ is a regular $(3^s,3^{s-1})$ negative Latin square type PDS in $\F_{3^{2s}}$, and
$f^{-1}(0) \sm \{0\}$ is a regular $(3^s,3^{s-1}-1)$ negative Latin square type PDS in $\F_{3^{2s}}$;
equivalently, $\big\{f^{-1}(1), f^{-1}(2)\big\}$ is a
$(3^{s-1},2)$ NLP-packing in~$(\F_{3^{2s}},+) = \Z_3^{2s}$.
\end{enumerate}

Following Remark~\ref{rem-identifyU}, we mention that in order to obtain a $(3^{s-1},3)$ LP-packing in the group $(\F_3^{2s},+) = \Z_3^{2s}$ relative to $\Z_3^s$ from the hyperbolic strongly regular bent function $f$ of $(i)$, we would firstly need to identify an order~$3^s$ subgroup $U$ of $\Z_3^{2s}$ contained in $f^{-1}(0)$, and secondly require $f^{-1}(0)-U$ to be a regular $(3^s,3^{s-1})$ Latin square type PDS in~$\Z_3^{2s}$. The first condition alone requires $f$ to be normal {\rm\cite[p.~106]{CMP}}, which in combination with the property that $f$ is weakly regular requires in turn that $f$ should be regular {\rm\cite[Thm.~6]{CMP}}.
This demonstrates that, in general, an LP-packing is a significantly more constrained structure than a hyperbolic strongly regular bent function.
\end{example}

\begin{remark}\label{rem-stronglyregular}
\quad
\begin{enumerate}[$(i)$]

\item
In {\rm \cite[Lemma 3.15]{CP13}}, an incomplete attempt was made to lift a hyperbolic strongly regular bent function on $\Z_p^{2a}$ to a hyperbolic strongly regular bent function on $\Z_{p^2}^{2a}$. We were able to accomplish this lifting in Theorem~\ref{thm-LPpackinguniform}, using the more constrained structure of an LP-packing. This underscores the importance of identifying the role of the subgroup $U$ in the definition of an LP-packing.

\item
{\rm Corollaries~\ref{coro-LPpackingeven}, ~\ref{coro-LPpackingodd} and~\ref{coro-NLPpacking}} recover the hyperbolic and elliptic strongly regular bent functions of {\rm\cite[Prop.~3.13]{CP13}} and strengthen those of {\rm\cite[Prop.~3.14]{CP13}}. The hyperbolic and elliptic strongly regular bent functions constructed in {\rm\cite[Prop.~3.12]{CP13}} have codomain of size~$2$, and are equivalent to reversible Hadamard difference sets. We examine this connection in more detail in Section~\ref{sec8} below.
\end{enumerate}
\end{remark}

\section{Relationship with reversible Hadamard difference sets}\label{sec8}
In this section, we examine how LP-packings and NLP-packings are closely related to reversible Hadamard difference sets.

\begin{definition}\label{def-HDS}
Let $G$ be a group $G$ of order~$v$.
A \emph{$(v,k,\la)$-difference set in $G$} is a $(v,k,\la,\la)$ partial difference set in~$G$.
A \emph{Hadamard difference set} has parameters
$(v,k,\la) = \big(4c^2,c(2c-1),c(c-1)\big)$ for some integer~$c$.
A difference set $D$ is \emph{reversible} if $D=D^{(-1)}$.
\end{definition}
\noindent
By Definition~\ref{def-pds}, a reversible $(v,k,\la)$-difference set that does not contain~$1_G$ is the same as a regular $(v,k,\la,\la)$ partial difference set.
The following result is a direct consequence of Definition~\ref{def-NLPpacking} and Lemmas~\ref{lem-pdschar} and~\ref{lem-LatinPDS}.

\begin{lemma}\label{lem-NLPrevHDS}
Let $c$ be a positive integer, and let $P$ be a $c(2c-1)$-subset of an abelian group $G$ of order~$4c^2$.
Then the following statements are equivalent.
\begin{enumerate}[$(i)$]
\item
$P$ is a $\big(4c^2,c(2c-1),c(c-1)\big)$ reversible Hadamard difference set in~$G$, and $1_G \notin P$
\item
$\chi(P) \in \{-c,c\}$ for all nonprincipal characters $\chi$ of $G$

\item
$\{G-P\}$ is a $(c,1)$ NLP-packing in~$G$.

\end{enumerate}
\end{lemma}

The next result follows from Definition~\ref{def-LPpacking} and Lemmas~\ref{lem-LatinPDS} and~\ref{lem-NLPrevHDS}.

\begin{lemma}\label{lem-LPpackingHDS}
Let $c$ be a positive integer,
let $G$ be an abelian group of order $4c^2$, and let $U$ be a subgroup of order~$2c$. Then $P_1,P_2$ are disjoint $\big(4c^2,c(2c-1),c(c-1)\big)$ reversible Hadamard difference sets in $G$ whose union is~$G-U$ if and only if $\{P_1,P_2\}$ is a $(c,2)$ LP-packing in $G$ relative to~$U$.
\end{lemma}

We now use LP-packings in $2$-groups to produce infinite families of reversible Hadamard difference sets.
\begin{corollary}\label{cor-revHDS}
Let $u,w$ and $s_3,\dots,s_v$ be nonnegative integers (not all zero).
For each $i =3,\dots,v$, let $G_i=\Z_{2^i}^{2s_i}$ and let $U_i$ be a subgroup of $G_i$ of order~$2^{is_i}$.
Let $G = \Z_2^{2u} \times \Z_4^w \times \prod_{i=3}^v G_i$, let $U = \Z_2^u \times (2\Z_4)^w \times \prod_{i=3}^v U_i$, and let $|G| = 4c^2$. Then

\begin{enumerate}[$(i)$]
\item
for $w \ne 1$, there exist two disjoint
$\big(4c^2,c(2c-1),c(c-1)\big)$ reversible Hadamard difference sets in $G$ whose union is~$G-U$.

\item
there exists a $\big(4c^2,c(2c-1),c(c-1)\big)$ reversible Hadamard difference set in~$G$.
\end{enumerate}
\end{corollary}

\begin{proof}
\quad
\begin{enumerate}[$(i)$]
\item
By Lemma~\ref{lem-LPpackingHDS}, an equivalent statement is that for $w \ne 1$ there exists a $(c,2)$ LP-packing in $G$ relative to~$U$. To show this when $w$ is even, take $p=2$ and $(s_1,s_2)=(u,\frac{w}{2})$ and $j=m-1 \ge 0$ in Corollary~\ref{coro-LPpackingeven}.
To show this when $w$ is odd, note that $\s(2,w) = \frac{w-1}{2} \ge 1$ and take
$p=2$ and $(s_1,s_2) = (u,0)$ and $j=m-1 \ge 0$ in Corollary~\ref{coro-LPpackingodd}.

\item
In view of part $(i)$, we may assume $w=1$.
In the case that $u$ and each of the $s_i$ is zero, there is a trivial reversible Hadamard difference set $D$ in $\Z_4$ comprising just the identity element.
Otherwise, take $p=2$ and $(s_1,s_2)=(u,0)$ and $j=m-1\ge 0$ in Corollary~\ref{coro-LPpackingeven} to produce a $(\frac{c}{2},2)$ LP-packing in
$\Z_2^{2u} \times \prod_{i=3}^v G_i$ relative to some subgroup of order~$c$. Use Theorem~\ref{thm-NLPLPpackingproduct} to combine this with the $(1,1)$ NLP-packing $\{\Z_4-D\}$ in $\Z_4$ (as given by Lemma~\ref{lem-NLPrevHDS}) to produce a $(c,1)$ NLP-packing in~$G$, then apply Lemma~\ref{lem-NLPrevHDS}.
\end{enumerate}
\end{proof}

\begin{remark}\label{rem-revHDS}
\quad
\begin{enumerate}[$(i)$]
\item
Corollary~\ref{cor-revHDS}~$(i)$ significantly extends a previous construction for $G=\Z_{2^a}^2$ and noncyclic $U$ of order~$2^a$ {\rm \cite[Thm.~3.1]{DP}}.

\item
All abelian $2$-groups in which a reversible Hadamard difference set is known to exist are recovered by Corollary~\ref{cor-revHDS}~$(ii)$; several other families of abelian groups are also known to contain reversible Hadamard difference sets (see {\rm\cite[Chap.~VI, Sect.~14]{BJL}}), and all these groups are recovered by {\rm \cite[Prop.~3.12]{CP13}}) using elliptic strongly regular bent functions.

\item
The construction of linking systems of reversible Hadamard difference sets in {\rm \cite[Thm.~5.3]{DMP}} depends on the existence of collections of disjoint partial difference sets in abelian $2$-groups satisfying mutual structural properties. After some modification, we can reinterpret this construction as combining a $(2^{(a-1)s},2^s)$ LP-packing in an abelian group of order $2^{2as}$, relative to some subgroup of order~$2^{as}$, with the complement of a $(2^s-1,2^{s-1}-1,2^{s-2}-1)$-difference set in $\Z_{2^s-1}$ {\rm \cite[Chap.~VI, Thm.~1.10]{BJL}} for $s > 1$ and $a > 0$.
\end{enumerate}
\end{remark}

\section{Open Problems}\label{sec9}
In this section, we propose some open problems for future research.

\begin{enumerate}[P1.]

\item
The recursive constructions of LP-packings described in Section~\ref{sec5} rest on two base cases: a $(1,p^s)$ LP-packing in $\Z_p^{2s}$ relative to $\Z_p^s$ derived from a spread (Example~\ref{ex-spread}), and the LP-packings in $p$-groups of odd rank given by~Result~\ref{res-HLX}. Find more base cases, especially in groups of odd rank (when a spread does not exist).

\item
Section~\ref{sec6} combines small examples of NLP-packings with the constructed families of LP-packings to produce families of NLP-packings. Find new small examples of NLP-packings.

\item
All known $(c,t)$ LP-packings with $t > 2$, and all known $(c,t)$ NLP-packings with $t>1$, occur in $p$-groups (see Section~\ref{sec8} for discussion of $(c,2)$ LP-packings and $(c,1)$ NLP-packings.)
Find examples in groups other than $p$-groups.

\item
Nonexistence results for LP-packings and NLP-packings can be obtained from known nonexistence results for single PDSs (see \cite{W}, for example).
We can also obtain nonexistence results for these packings from known nonexistence results for reversible Hadamard difference sets \cite[Chap.~VI, Sect.~14]{BJL}, by applying Lemmas~\ref{lem-LPpackingunion} and~\ref{lem-LPpackingHDS} to $(c,2t)$ LP-packings, and Lemmas~\ref{lem-NLPpackingunion} and~\ref{lem-NLPrevHDS} to $(c,2t-1)$ NLP-packings. Establish further nonexistence results for LP-packings and NLP-packings.

\item
Construct hyperbolic strongly regular bent functions satisfying additional conditions in order to produce LP-packings in new groups (see Remark~\ref{rem-identifyU}).

\item
Are there collections of partial difference sets with advantageous structure, other than LP-packings and NLP-packings, that can be analyzed using methods similar to those of this paper?

\end{enumerate}

\section*{Acknowledgements}
The authors gratefully acknowledge helpful discussions with Samuel Simon at the outset of this research.
We thank Alexander Pott for his hospitality when the first author visited the second at Otto von Guericke University, Magdeburg in Summer 2018 to pursue this research. We thank the anonymous reviewers for their detailed and helpful comments.

\bibliographystyle{amsplain}
\providecommand{\bysame}{\leavevmode\hbox to3em{\hrulefill}\thinspace}
\providecommand{\MRno}{\relax\ifhmode\unskip\space\fi MR }
\providecommand{\MRhref}[2]{%
  \href{http://www.ams.org/mathscinet-getitem?mr=#1}{#2}
}
\providecommand{\href}[2]{#2}
\providecommand{\MR}[1]{\MRhref{#1}{\MRno{#1}}}

\end{document}